\documentclass[12pt, reqno]{amsart}

\usepackage[a4paper,top=3cm,bottom=2.5cm,left=2.75cm,right=2.75cm,marginparwidth=1.75cm]{geometry}

\usepackage{amsmath}
\usepackage{amsthm}
\usepackage{amsfonts}
\usepackage{amssymb}
\usepackage{mathtools}
\usepackage[mathscr]{euscript}
\usepackage{bm}
\usepackage[table]{xcolor}
\usepackage[all,cmtip]{xy}
\usepackage{enumitem}
\usepackage{tikz}  
\usepackage{tikz-cd}
\usepackage{graphicx}
\usepackage{scalerel}
\usepackage{microtype}
\usepackage{comment}

\usepackage{verbatim}
\usepackage{libertine}

\usepackage[pdftex,
              pdfauthor={},
              pdftitle={},
              pdfsubject={},
              pdfkeywords={},
              colorlinks=true, 
              citecolor=citation, 
              urlcolor=citation, 
              linkcolor=reference]{hyperref}

\usepackage{cleveref} 

\usepackage[T2A,T1]{fontenc}
\DeclareSymbolFont{cyrillic}{T2A}{cmr}{m}{n}
\DeclareMathSymbol{\Sha}{\mathalpha}{cyrillic}{216}

\makeatletter 
\@mparswitchfalse%
\makeatother
\normalmarginpar 

\usepackage{comment}
\usepackage{array}
\usepackage{tikz-cd}

\definecolor{citation}{rgb}{0,.40,.80}
\definecolor{reference}{rgb}{.80,0,.40}
\definecolor{todogray}{HTML}{E4E4E4}

\numberwithin{equation}{section}

\theoremstyle{plain}
\newtheorem{theorem}{Theorem}[section]
\newtheorem{lemma}[theorem]{Lemma}
\newtheorem{proposition}[theorem]{Proposition}

\newtheorem{question}[theorem]{Question}

\theoremstyle{definition}
\newtheorem{definition}[theorem]{Definition}

\newtheorem{remark}[theorem]{Remark}

\newtheorem*{remark*}{Remark}

\newcommand{\Db}{\mathrm{D^b}}

\newcommand{\Gr}{\mathrm{Gr}}

\DeclareMathOperator{\Pic}{Pic}
\DeclareMathOperator{\NS}{NS}

\newcommand{\functionstar}[4]{
\begin{array}{rcl} #1 &\longrightarrow &#2 \\ #3&\longmapsto &#4 \end{array}
}

\newcommand{\function}[5]{\begin{array}{lrcl} 
#1: & #2 & \longrightarrow & #3 \\
    & #4 & \longmapsto & #5 \end{array}}

\DeclareMathOperator{\Hom}{Hom}

\DeclareMathOperator{\Ext}{Ext}
\DeclareMathOperator{\Aut}{Aut}

\DeclareMathOperator{\Bl}{Bl}

\newcommand{\id}{\mathrm{id}}

\newcommand{\sgn}{\mathrm{sgn}}
\newcommand{\Lts}{\ensuremath{L_{2\text{-}6}}}
\newcommand{\Lte}{\ensuremath{L_{2\text{-}8}}}
\newcommand{\Lto}{\ensuremath{L_{3\text{-}1}}}
\newcommand{\Ato}{\ensuremath{A_{3\text{-}1}}}

\DeclareMathOperator{\FM}{FM}

\DeclareMathOperator{\Jac}{J}

\DeclareMathOperator{\Br}{Br}

\newcommand{\disc}{\mathrm{disc}}

\newcommand{\rk}{\mathrm{rk}}

\docsvlist{A,B,C,D,E,F,G,H,I,J,K,L,M,N,O,P,Q,R,S,T,U,V,W,X,Y,Z} 

\docsvlist{A,B,C,D,E,F,G,H,I,J,K,L,M,N,O,P,Q,R,S,T,U,V,W,X,Y,Z} 

\newcommand{\bC}{\mathbf{C}}

\newcommand{\bZ}{\mathbf{Z}}
\newcommand{\bP}{\mathbf{P}}
\newcommand{\bQ}{\mathbf{Q}}
\newcommand{\bR}{\mathbf{R}}
\newcommand{\bL}{\mathbf{L}}

\usepackage[numbers]{natbib}
\begin{document}

\title[Categorical Torelli Theorems for higher Picard rank Fano Double Covers]{Categorical Torelli Theorems for higher Picard rank Fano Double Covers}

\author[Augustinas Jacovskis and Reinder Meinsma]{Augustinas Jacovskis and Reinder Meinsma}
\address{AJ: Nacionalinis Fizinių ir Technologijos Mokslų Centras, 
Saulėtekio al. 3, Vilnius, 10257, Lithuania}
\email{\texttt{augustinas.jacovskis@ff.vu.lt}}
\address{RM: Fakultät für Mathematik und Informatik,
Universität des Saarlandes,
Campus E2.4, 66123 Saarbrücken, Germany
}
\email{\texttt{meinsma@math.uni-sb.de}}

\begin{abstract}
We prove categorical Torelli theorems for four families of Fano double covers with Picard rank $\rho >1$. Among these is the family of Verra fourfolds. The other three families manifest as double covers of Fano threefolds, branched in anticanonical K3 surfaces. For the three families of threefolds, our proof is based on reducing equivalences between the Kuznetsov components of Fano threefolds in the same deformation family to derived equivalences of their respective K3 branch divisors, and deducing that the resulting isomorphism of branch divisors gives rise to an isomorphism of the Fano threefolds for each family. For Verra fourfolds, we show that an equivalence of their Kuznetsov components induces an isomorphism of the branch divisors using the theory of $2$-torsion Brauer classes on K3 surfaces.
\end{abstract}

\maketitle

\tableofcontents

\section{Introduction}
\addtocontents{toc}{\protect\setcounter{tocdepth}{1}}
\renewcommand{\thetheorem}{\Alph{theorem}}

Let $\Db(X)$ be the bounded derived category of coherent sheaves on a variety $X$. The notion of a \emph{subcategory} of $\Db(X)$ determining $X$ is known as a \emph{categorical Torelli theorem}. In the setting of smooth Fano threefolds of Picard rank $1$, the question has been settled for almost all deformation families; see \cite{PS} for a survey. The subcategory in question in this context is known as the \emph{Kuznetsov component} (see e.g. \cite{KuzSOD}). 

In the present paper, we prove categorical Torelli theorems for four deformation families of Picard rank $ \rho \coloneqq \rk \Pic X > 1$ Fano varieties over the field of the complex numbers $\bC$. The members of the families are given as double covers $f \colon X \to Y$ branched in a divisor $Z$, where 

\begin{enumerate}
    \item \textbf{Family 2-6(b):} $Y$ is a divisor of $\bP^2 \times \bP^2$ of bidegree $(1,1)$ and $Z$ is an anticanonical divisor. These are also known as \emph{special Verra threefolds};
    \item \textbf{Family 2-8:} $Y = \Bl_p \bP^3$ and $Z$ is an anticanonical divisor, such that the intersection of $Z$ and the exceptional divisor $\iota \colon E \to Y$ is smooth;
    \item \textbf{Family 3-1:} $Y= \bP^1 \times \bP^1 \times \bP^1$ and $Z$ is a divisor of tridegree $(2,2,2)$;
    \item \textbf{Verra fourfolds:} $Y = \bP^2 \times \bP^2$ and $Z$ is an \textit{ordinary Verra threefold}, that is, $Z$ is a divisor of bidegree $(2,2)$ in $Y$.
\end{enumerate}

The Kuznetsov components $\cA_X$ of the four families can be defined as follows.
The base $Y$ of a Fano threefold $X$ in Family 2-6(b) can be equivalently expressed as the projective bundle $Y\simeq \bP(T_{\bP^2})$, and the derived category of $X$ has the following semiorthogonal decomposition (see (\ref{eq: 2-6(b) Kuznetsov component}))
 \begin{equation*}
     \Db(X) = \langle \cA_X, \pi^* \Db(\bP^2) \rangle
 \end{equation*}
 where $\pi \colon \bP(T_{\bP^2}) \to \bP^2$ is the projection. 
 
 The Fano threefolds of Family 2-8 have semiorthogonal decompositions (see Lemma \ref{lem:blow_up_sod})
\begin{equation*}
    \Db(X) = \langle \cA_X, \cO_X(-H), \cO_X(-E), \cO_X  \rangle,
\end{equation*}
where $H$ and $E$ are the class of the hyperplane and exceptional divisor on $Y$, respectively. 

The Fano threefolds of Family 3-1 have semiorthogonal decompositions (see Lemma \ref{lem:3-1 A_X SOD})
\begin{equation*}
    \Db(X) = \langle \cA_X, \cO_X(0,0,0), \cO_X(1,0,0), \cO_X(0,1,0) , \cO_X(0,0,1) \rangle ,
\end{equation*}
where e.g. $\cO_X(1,0,0) = \cO_{\bP^1}(1) \boxtimes \cO_{\bP^1} \boxtimes \cO_{\bP^1}$. 

Finally Verra fourfolds have the following semiorthogonal decomposition (see Section \ref{sec: Verra fourfolds}): 
\begin{equation*}
    \Db(X) = \langle \cA_X, \cO_{X}(-1,0), \cO_{X}(0,0) , \cO_{X}(1,0) ,
        \cO_{X}(0,1), \cO_{X}(1,1) , \cO_{X}(2,1) \rangle.
\end{equation*}

The subcategories $\cA_X$ in each case are defined as the right orthogonals to the exceptional collections on their right hand sides, and these subcategories are known as the \emph{Kuznetsov components} of each $X$, respectively. Our main theorems are as follows:

\begin{theorem}[= Theorem \ref{thm: main theorem}] \label{thm: intro main theorem}
    Let $X$ and $X'$ be very general Fano threefolds both of deformation type either 2-6(b), 2-8, or 3-1. Then an equivalence of Kuznetsov components $\cA_X \simeq \cA_{X'}$ 
    implies an isomorphism between $X$ and $X'$. 
\end{theorem}

In the statement of the theorem later in the paper, we use a more precise notion of generality (see Definition \ref{def: Z general}). 

\begin{theorem}[= Theorem \ref{thm: categorical torelli for verra fourfolds}]\label{thm: categorical torelli for verra fourfolds INTRO}
    Let $V_4$ and $V_4'$ be very general Verra fourfolds. Then an equivalence of Kuznetsov components $\cA_{V_4} \simeq \cA_{V_4'}$ 
    implies an isomorphism between $V_4$ and $V_4'$.
\end{theorem}

\subsection*{Sketch of the proof}
To prove Theorem \ref{thm: intro main theorem}, we descend the equivalence of Kuznetsov components to an equivalence of $\mu_2$-equivariant Kuznetsov components: $$\cA_X\simeq \cA_{X'} \implies \cA_X^{\mu_2} \simeq \cA_{X'}^{\mu_2},$$ where $\mu_2$ is generated by the canonical involution of the double cover. We show this by proving that in each case, we have an isomorphism of functors
\[
    S_{\cA_X} \simeq \tau_{\cA_X}[2],
\]
where $\tau$ denotes the pullback along the covering involution, and $S_{\cA_X}$ is the Serre functor of the Kuznetsov component (see Theorem \ref{thm: any equivalence lifts to equivariant categories}). This result uses the techniques from \cite{IK}.

In Theorem \ref{thm: equivariant kuznetsov components are the branch loci}, we deduce that the equivariant Kuznetsov component is equivalent to the derived category of the branch divisor: 
\[
\cA_{X}^{\mu_2}\simeq \Db(Z), \quad   \cA_{X'}^{\mu_2}\simeq \Db(Z')
\]
We establish Theorem \ref{thm: equivariant kuznetsov components are the branch loci} with a study of several semiorthogonal decompositions of the equivariant derived category $\Db(X)^{\mu_2}$ that arise from the geometry of $X$.

The next step is to study Fourier--Mukai partners of the branch divisors. Since the branch divisors are K3 surfaces, this can be done using the Derived Torelli Theorem for K3 surfaces \cite{Muk,OrlovK3}. The Derived Torelli Theorem was used in \cite{HLOY} and \cite{MS24} to establish counting formulas for Fourier--Mukai partners of K3 surfaces. We use these counting formulas in Theorem \ref{thm: the k3s have no FM partners} to show that the branch divisors have no non-trivial Fourier--Mukai partners:
\[
 \Db(Z) \simeq \Db(Z') \implies Z\simeq Z'.
\]

The final piece of the proof is Theorem \ref{thm: branch determines Fano}, which shows that the isomorphism class of the branch divisor $Z$ determines $X$ up to isomorphism:
\[
Z\simeq Z' \implies X\simeq X'.
\]
We prove this via an analysis of the geometry of the base $Y$ of the double cover, combined with a study of the Neron--Severi lattices of the branch divisors. We put this all together in Section \ref{sec:categorical torelli theorems}.

Our proof of Theorem \ref{thm: categorical torelli for verra fourfolds INTRO} skips a few of the above steps, as we can directly prove the following implication in Proposition \ref{prop: the derived category of the verra twisted K3 determines the Verra threefold}:
\[
\cA_{V_4} \simeq \cA_{V_4'}\implies V_3 \simeq V_3'.
\]
Here, $V_3$ and $V_3'$ denote the branch loci of the Verra fourfolds $V_4$ and $V_4'$, respectively. The proof of this implication uses the theory of $2$-torsion Brauer classes on K3 surfaces of degree $2$ developed by Van Geemen in \cite{vGe}. 

Similarly to the other three families, the next step of the proof is to show the following (see Proposition \ref{prop: branch determines cover V4}):
\[
V_3 \simeq V_3'\implies V_4 \simeq V_4'.
\]
This finishes the proof of Theorem \ref{thm: categorical torelli for verra fourfolds INTRO}. 

\subsection*{Relation to previous work}

A similar approach was considered for special Gushel--Mukai threefolds in \cite[Theorem 9.9]{JLLZ}. These are double covers of a linear section of the Grassmannian $\Gr(2,5)$ branched in a K3 surface. In the special Gushel--Mukai case, the equivariant Kuznetsov component is also equivalent to the derived category of the branch K3. However, since the Picard rank of a special Gushel--Mukai threefold is $1$, the counting of Fourier--Mukai partners, as well as showing that the branch divisor determines the cover is much simpler. The deformation families of Fano threefolds (and in particular their branch divisors) considered in the present paper have Picard rank greater than $1$, and the bases of two out of three of the families have non-rectangular Lefschetz decompositions. This necessitates a more subtle analysis. 

Fano threefolds of Picard rank $1$ manifesting as double covers with \emph{canonically polarised} branch divisors were considered in \cite{DJR}. Due to the equivariant Kuznetsov component containing (rather than being equivalent to) the derived category of the branch divisor in that context, the analogous equivalences of equivariant Kuznetsov components in those cases were reduced to Hodge isometries of the middle primitive cohomologies of the branch divisors. Classical Torelli theorems were then employed to yield the desired categorical Torelli theorems. 

\begin{remark*}
    There are two other deformation families of smooth Fano threefolds of Picard rank $\rho \geq 2$ which are listed in \cite{fanography} as double covers that we do not consider, since the methods of the present paper do not apply to them (see Section \ref{sec: remaining double covers}).
\end{remark*}

\subsection*{Notation and conventions} We work over the field of complex numbers $\bC$. All functors are assumed to be derived. 

\subsection*{Acknowledgements} We would like to thank Pieter Belmans for suggesting the case of higher Picard rank Fano threefolds, Sasha Kuznetsov for pointing out the reference \cite{IK}, and Francesco Denisi and Yuta Takada for interesting and helpful discussions. We thank Jason Starr for answering a question regarding Family 2-18. Finally, we warmly thank Benedetta Piroddi for many helpful discussions and interest in our work. AJ was supported by Research Council of Lithuania
Grant No. P-PD-25-081. This work is a contribution to Project A22 in SFB 195 No. 286237555 of DFG.

\setcounter{theorem}{0}
\renewcommand{\thetheorem}{\thesection.\arabic{theorem}}

\section{Background} \label{sec:background}
\subsection{Derived categories}

For objects $E, F$ in a triangulated category $\cC$, we write
\begin{equation*}
    \Hom^\bullet(E, F) \coloneqq \bigoplus_{i \in \bZ} \Ext^i(E, F)[-i] .
\end{equation*}
A subcategory $\cA \subset \cC$ is called \emph{admissible} if both the left and right adjoint (denoted $i^!$ and $i^*$, respectively) to the inclusion $i \colon \cA \hookrightarrow \cC$ exist. An object $E \in \cC$ is called \emph{exceptional} if $\Hom^\bullet(E,E) = \bC[0]$.

\begin{definition}
    We say $\cC = \langle \cA_1, \dots , \cA_n \rangle$ is a \emph{semiorthogonal decomposition} of a triangulated category $\cC$ if
    \begin{enumerate}
        \item $\Hom^\bullet(F, G) = 0$ for all $F \in \cA_i, G \in \cA_j$ if $i > j$;
        \item for any $F \in \cC$, there exists a sequence of morphisms
        \begin{equation*}
        0 = F_n \to F_{n-1} \to \cdots \to F_1 \to F_0 = F
        \end{equation*}
        such that $\mathrm{Cone}(F_i \to F_{i-1}) \in \cA_i$.
    \end{enumerate}
    When the subcategories $\cA_i$ are all exceptional objects satisfying condition (1) above, we call $\{ \cA_1, \dots , \cA_n \}$ an \emph{exceptional collection}. An exceptional collection in a category is called \emph{full} if it generates the category. 
\end{definition}

\begin{lemma} \label{lem:SODs and Serre functors}
    If $\cC = \langle \cA_1, \cA_2 \rangle $ is a semiorthogonal decomposition, then so are 
    \begin{equation*}
        \cC = \langle S_{\cC}(\cA_2), \cA_1 \rangle = \langle \cA_2, S_{\cC}^{-1}(\cA_1)  \rangle ,
    \end{equation*}
    where $S_{\cC}$ is the Serre functor of $\cC$.
\end{lemma}

\begin{definition}
    Let $\cA \subset \cC$ be a subcategory. Define the \emph{right orthogonal} $\cA^\perp$ of $\cA$ inside $\cC$ to be the subcategory
    \begin{equation*}
        \cA^\perp \coloneqq \{  F \in \cC \mid \Hom^\bullet(G, F) = 0 \text{ for all } G \in \cA  \}
    \end{equation*}
    and the \emph{left orthogonal} $\cA^\perp$ of $\cA$ inside $\cC$ to be the subcategory
    \begin{equation*}
        {}^\perp\cA \coloneqq \{  F \in \cC \mid \Hom^\bullet(F, G) = 0 \text{ for all } G \in \cA  \} .
    \end{equation*}
\end{definition}

\begin{definition}
        Let $i \colon \cA \hookrightarrow \cC$ be an admissible subcategory, and $F \in \cC$ an object. We define the 
    \emph{left mutation of $F$ along $\cA$}, denoted $\bL_\cA F$, by the triangle 
    \begin{equation*}
        ii^!F \to F \to \bL_\cA F
    \end{equation*}
    and the \emph{right mutation of $F$ along $\cA$}, denoted $\bR_\cA F$, by the triangle 
    \begin{equation*}
        \bR_\cA F \to F \to ii^* F .
    \end{equation*}
    In particular, when the admissible subcategory $\cA$ is a single exceptional object $E$, the triangles defining the left and right mutations along $E$ become 
    \begin{equation*}
        E \otimes \Hom^\bullet(E, F) \to F \to \bL_E F
    \end{equation*}
    and
    \begin{equation*}
        \bR_E F \to F \to E \otimes \Hom^\bullet(F, E)^\vee,
    \end{equation*}
    respectively.
\end{definition}

\begin{definition}\label{def: rectangular lefschetz decomposition}
    Let $X$ be a smooth projective variety, and let $H$ be an ample divisor of $X$. The derived category $\Db(X)$ is said to have a \emph{rectangular Lefschetz decomposition} of length $n$ if it admits a semiorthogonal decomposition
    \begin{equation*}
        \Db(X) = \langle \cA, \cA(1), \dots, \cA(n)  \rangle 
    \end{equation*}
    where $\cA \subset \cC$ is an admissible subcategory, and $\cA(i) \coloneqq \cA \otimes \cO_X(iH)$.
\end{definition}

\subsection{Equivariant categories}

For this section, we mostly follow \cite[Section 3]{KP}.

\begin{definition}
    Let $\cC$ be a category, and $G$ a finite group. An \emph{action of $G$ on $\cC$} is the following data:
    \begin{enumerate}
        \item For every $g \in G$, and autoequivalence $g^* \colon \cC \xrightarrow{\sim} \cC$;
        \item For every $g, h \in G$, and isomorphism of functors $c_{g,h} \colon (gh)^* \xrightarrow{\sim} h^* \circ g^*$, such that the diagram
        \begin{equation*}
\begin{tikzcd}
	{(fgh)^*} & {h^* \circ (fg)^*} \\
	{(gh)^* \circ f^*} & {h^* \circ g^* \circ f^*}
	\arrow["{c_{fg, h}}", from=1-1, to=1-2]
	\arrow["{c_{f, gh}}"', from=1-1, to=2-1]
	\arrow["{h^* c_{f,g}}", from=1-2, to=2-2]
	\arrow["{c_{g,h} f^*}", from=2-1, to=2-2]
\end{tikzcd}
        \end{equation*}
        commutes for all $f,g,h \in G$.
    \end{enumerate}
\end{definition}

\begin{definition}
    Let $G$ be an action on $\cC$. A \emph{$G$-equivariant object} of $\cC$ is a pair $(F, \phi)$, where $F$ is an object of $\cC$ and $\phi = \{ \phi_g \}_{g \in G}$ is a collection of isomorphisms $\phi_g \colon F \to g^* F$ for $g \in G$, such that the diagram
    \begin{equation*}
        \begin{tikzcd}
	F & {h^*(F)} & {h^*(g^*(F))} \\
	&& {(gh)^*(F)}
	\arrow["{\phi_h}", from=1-1, to=1-2]
	\arrow["{\phi_{gh}}"', from=1-1, to=2-3]
	\arrow["{h^*(\phi_g)}", from=1-2, to=1-3]
	\arrow["{c_{g,h}(F)}"', from=2-3, to=1-3]
\end{tikzcd}
    \end{equation*}
    commutes for all $g, h \in G$. The \emph{$G$-equivariant category} $\cC^G$ of $\cC$ is the category of $G$-equivariant objects of $\cC$.
\end{definition}

If $\cC$ is the derived category of a variety $X$, or if $\cC$ is a semiorthogonal component of $\Db(X)$, then $\cC^G$ is also triangulated. 

\begin{theorem}[{\cite[Theorem 6.3]{ElaginDescent}}] \label{thm:G action commutes through SOD}
    Let $X$ be a quasi-projective variety with an action of a finite group $G$. Assume that $\Db(X) = \langle \cA_1, \dots, \cA_n \rangle$ is a semiorthogonal decomposition preserved by $G$, i.e. each $\cA_i$ is preserved by the action of $G$. Then there is a semiorthogonal decomposition
    \begin{equation*}
        \Db(X)^G = \langle \cA_1^G, \dots, \cA_n^G \rangle.  
    \end{equation*}
\end{theorem}

An action of a group $G$ on a category $\cC$ is said to be \emph{trivial} if for each $g \in G$, we have an isomorphism of functors $t_g \colon \id \xrightarrow{\sim} g^*$ such that 
\begin{equation*}
    c_{g,h} \circ t_{gh} = h^* t_g \circ t_h
\end{equation*}
for all $g, h \in G$.

\begin{proposition}[{\cite[Proposition 3.3]{KP}}] \label{prop:trivial action SOD}
    Let $\cC$ be a triangulated category with a trivial action of a finite group $G$. Then $\cC^G$ is also a triangulated category with a semiorthogonal decomposition
    \begin{equation*}
        \cC^G = \langle \cC \otimes \rho_0 , \dots , \cC \otimes \rho_n \rangle , 
    \end{equation*}
    where $\rho_1, \dots, \rho_n$ are the finite-dimensional irreducible representations of $G$.
\end{proposition}

We will often abbreviate $\cC \otimes \rho_i$ as $\cC \rho_i$.

\begin{definition} \label{def: FM transform} \leavevmode
    Let $X$ and $Y$ be schemes. A functor $\Phi_{\cE} \colon \Db(X) \to \Db(Y)$ given by $\Phi_{\cE}(F) \coloneqq p_{Y *}(p_X^*(F) \otimes \cE)$ is called a \emph{Fourier--Mukai} functor with \emph{Fourier--Mukai kernel} $\cE \in \Db(X \times Y)$, where $p_X$ and $p_Y$ are the obvious projections.
\end{definition}

\subsection{Generalities on derived categories of cyclic covers}

Let $f \colon X \to Y $ be a double cover branched in $i \colon Z \hookrightarrow Y$. Denote the embedding of $Z$ as the ramification divisor $j \colon Z \hookrightarrow X$. Let $\mu_2 = \{ 1, -1 \}$ be the group generated by the involution associated to the double cover. Let $\widehat{\mu}_2 \simeq \bZ / 2 \bZ $ be the dual group of $\mu_2$. Denote by $\rho_0, \rho_1 \in \widehat{\mu}_2$ the irreducible representations of $\mu_2$, which correspond to the primitive characters of $\widehat{\mu}_2$. Let 
\begin{equation}
    f_k^* \colon \Db(Y) \xrightarrow{- \otimes \rho_k} \Db(Y)^{\mu_2} \xrightarrow{f^*} \Db(X)^{\mu_2}
\end{equation}
and 
\begin{equation}
    j_{k *} \colon \Db(Z) \xrightarrow{- \otimes \rho_k} \Db(Z)^{\mu_2} \xrightarrow{j_*} \Db(X)^{\mu_2} .
\end{equation}

\begin{theorem}[{\cite[Theorem 5.1]{Gluing}, \cite[Theorem 1.6]{McKay}, \cite[Theorem 4.1]{KP}}] \label{thm: KP thm 4.1}
    For $k =0, 1$, the functors $f_k^*$ and $j_{k*}$ are fully faithful. Moreover, we have a semiorthogonal decomposition
    \begin{equation*}
        \Db(X)^{\mu_2} = \langle f_0^* \Db(Y), j_{0*} \Db(Z)  \rangle .
    \end{equation*}
\end{theorem}

\begin{lemma}[{\cite[Lemma 5.1]{KP}}] \label{lem: SOD from rectangular lefschetz base}
    Let $Y$ be a smooth projective variety with a rectangular Lefschetz decomposition $\Db(Y) = \langle \cB, \cB(1), \dots , \cB(m-1) \rangle$. Let $d$ be a positive integer such that $d < m$, and let $f \colon X \to Y$ be a double cover branched in a divisor in $|\cO_Y(2d)|$.
    \begin{enumerate}
        \item The functor $f^* \colon \Db(Y) \to \Db(X)$ is fully faithful on the subcategory $\cB \subset \Db(Y)$.
        \item Let $\cB_X \coloneqq f^* \cB$. There is a semiorthogonal decomposition
        \begin{equation*}
            \Db(X) = \langle  \cA_X , \cB_X , \cB_X(1), \dots , \cB_X(m-d-1) \rangle 
        \end{equation*}
        where the \emph{Kuznetsov component} $\cA_X$ is the right orthogonal to $$ \langle \cB_X , \cB_X(1), \dots , \cB_X(m-d-1) \rangle.$$
    \end{enumerate}
\end{lemma}

When $\Db(Y)$ contains a (not necessarily full) exceptional collection and is not necessarily rectangular Lefschetz, we have the following.  

\begin{proposition}[{\cite[Lemma 3.1 and Proposition 3.2]{IK}}] \label{prop: double cover SODs IK}
    Let $f \colon X \to Y$ be as above, and let $Z$ be the branch divisor. Suppose for some $H \in \Pic Y$, we have $Z \sim 2H \sim -K_Y$. Assume $\underline{\cE} =  \{ \cE_1, \dots, \cE_n  \}$ is an exceptional collection in $\Db(Y)$ such that the collection
    \begin{equation} \label{eq: E(-H), E}
        \{ \underline{\cE}(-H), \underline{\cE} \}
    \end{equation}
    is also exceptional in $\Db(Y)$. Then the collection $f^* \underline{\cE} \in \Db(X)$ is exceptional in $\Db(X)$, and in particular, 
    \begin{equation*}
        \Db(X) = \langle \cA_X , f^*\underline{\cE} \rangle
    \end{equation*}
    is a semiorthogonal decomposition, where the \emph{Kuznetsov component} $\cA_X$ is the right orthogonal to the exceptional collection $f^* \underline{\cE}$.
\end{proposition}

Adding one extra assumption gives a convenient description of the Serre functor of the Kuznetsov component.

\begin{proposition}[{\cite[Proposition 3.2]{IK}}] \label{prop: IK prop 3.2}
    Suppose we are in the setting of Proposition \ref{prop: double cover SODs IK}, and further assume that the exceptional collection (\ref{eq: E(-H), E}) is full in $\Db(Y)$. Let $\tau_{\cA_X}$ be the involution functor induced by the involution $\tau$ of the double cover. Then the Serre functor of $\cA_X$ is given by
    \begin{equation*}
        S_{\cA_X} \simeq \tau_{\cA_X}[\dim X -1].
    \end{equation*}
\end{proposition}

\subsection{Derived categories of quadric fibrations}

For the study of Verra fourfolds and ordinary Verra threefolds later in the paper, we will require a result on the semiorthogonal decompositions of quadric bundles. 

\begin{theorem}[{\cite[Theorem 4.2]{KuzQuad}}] \label{thm: Kuz quadric fibrations SOD}
    Let $p \colon X \to S$ be a flat quadric fibration of relative dimension $n-2$. Then there is a semiorthogonal decomposition
    \begin{equation*}
        \Db(X) = \langle \Db(S, \cB_0), p^* (\Db(S)) \otimes \cO_{X/S}(1), \dots, p^* (\Db(S)) \otimes \cO_{X/S}(n-2) \rangle .
    \end{equation*}
\end{theorem}

Here, $\cB_0$ is the sheaf of even parts of Clifford algebras on $S$, see \cite[Section 3]{KuzQuad}.

\subsection{Lattices}
In this subsection, we recall some facts about lattices necessary to study the N\'eron--Severi lattices of the branch divisors of our Fano threefolds. 
Our main reference for lattice theory is \cite{Nik80}.

A \textit{lattice} is a free, finitely generated abelian group $L$ equipped with a non-degenerate integral symmetric bilinear form $b\colon L\times L \to \bZ$. For $v,u\in L$, we usually write $v\cdot u \coloneqq b(v,u)$ and $v^2 \coloneqq b(v,v)$. An isomorphism of lattices is called an \textit{isometry.} The group of isometries of $L$ is denoted $O(L)$.
The \textit{dual} of a lattice $L$ is the group $L^* \coloneqq \Hom(L,\bZ)$. The bilinear form on $L$ induces an injective group homomorphism 
\[
\functionstar{L}{L^*}{v}{b(v,-).}
\]
The cokernel $A_L\coloneqq L^*/L$ is called the \textit{discriminant group} of $L$, and the order of the discriminant group $$\disc(L)\coloneqq |A_L|$$ is called the \textit{discriminant} of $L$. We say that $L$ is \textit{unimodular} if $\disc(L)=1$. 
Since the dual can be characterised as $$L^* \simeq \left\{x\in L\otimes \bQ\mid \forall v\in L: x\cdot v\in \bZ\right\},$$ 
it follows that $L^*$ inherits a non-degenerate symmetric bilinear form taking values in $\bQ$. If we assume our lattice $L$ to be \textit{even}, that is, we have $v^2 \in 2\bZ$ for all $v \in L$, then $A_L$ inherits a quadratic form $q\colon A_L \to \bQ/2\bZ$ from $L^*$. Note that any isometry $\sigma \in O(L)$ naturally induces an isometry $\overline{\sigma}\in O(A_L)$.

The \textit{divisibility} of an element $v\in L$ is the positive integer
\begin{equation}\label{eq: divisibility}
\mathrm{div}(v) \coloneqq \gcd_{u\in L}(v\cdot u).
\end{equation}

For a lattice $L$, we define the \textit{signature} $\sgn(L) = (r,s)$ of $L$ to be the signature of the real quadratic space $L\otimes \bR$. We say $L$ is \textit{indefinite} if $r,s>0$. 

\begin{definition}\label{def: genus}
The \textit{genus} of a lattice $L$ is the set, denoted $\cG(L)$, of isomorphism classes of lattices $L'$ that satisfy the following two conditions:
\begin{enumerate}
    \item We have $\sgn(L) = \sgn(L')$.
    \item There is an isomorphism $A_L\simeq A_{L'}$ respecting the natural quadratic forms on $A_L$ and $A_{L'}$. 
\end{enumerate} 
    
\end{definition}

An embedding of lattices $N\hookrightarrow L$ is said to be \textit{primitive} if the cokernel $L/N$ is torsion-free. If, in addition, $L$ is unimodular, we have two natural isomorphisms of groups:
\[
\xymatrix{
    &\frac{L}{N\oplus N^\perp} \ar[dr]_{f_{N^\perp}}^\simeq \ar[dl]^{f_{N}}_\simeq & \\
    A_N&& A_{N^\perp}.
}
\]
Explicitly, an element $x\in L$ induces an element $(x\cdot -)|_N\in N^*$, and the homomorphism $f_N$ sends the class of $x$ in $\frac{L}{N\oplus N^\perp}$ to the class of $(x\cdot-)|_N$ in $A_N$.
The group isomorphism $f_{N^\perp}\circ f_N^{-1}$ is an isometry if we multiply the quadratic form on $A_{N^\perp}$ by $-1$, i.e. we have an isometry $$f_{N^\perp}\circ f_N^{-1}\colon A_N \simeq A_{N^\perp}(-1).$$ 

\begin{lemma}[{\cite[Proposition 1.6.1]{Nik80}}]  \label{lem: extending isometries} 
    Let $N\hookrightarrow L$ be a primitive embedding of even lattices, and suppose $L$ is unimodular. Let $\sigma \in O(N)$. Then there exists an isometry $\tilde{\sigma}\in O(L)$ such that the diagram 
    \[
    \xymatrix{
        N \ar@{^(->}[r] \ar[d]_\sigma^\simeq& L \ar[d]^{\tilde{\sigma}}_\simeq \\
        N \ar@{^(->}[r]& L 
    }
    \]
    commutes, if and only if there is an isometry $\tau\in O(N^\perp)$ such that the diagram 
    \[
    \xymatrix{
        A_N \ar[r]^\simeq \ar[d]_{\overline{\sigma}}^\simeq& A_{N^\perp}(-1) \ar[d]^{\overline{\tau}}_\simeq \\
        A_N \ar[r]^\simeq& A_{N^\perp}(-1)
    }
    \]
    commutes.
\end{lemma}

An embedding of even lattices $N\hookrightarrow L$ is said to be an even \textit{overlattice} if the cokernel $L/N$ is a finite group. In this case, we have a sequence of finite-index embeddings
\[
N\hookrightarrow L \hookrightarrow L^* \hookrightarrow N^*.
\]
Therefore, the quotient $H_L\coloneqq L/N$ comes with a natural embedding into $A_N = N^*/N$. 

\begin{lemma}[{\cite[Proposition 1.4.2]{Nik80}}] \label{lem: overlattice subgroup correspondence} 
    Let $N\hookrightarrow L$ and $N'\hookrightarrow L'$ be even overlattices. Then an isometry $\sigma \colon N\simeq N'$ can be extended to an isometry $\tilde{\sigma}\colon L\simeq L'$ if and only if $$\overline{\sigma}(H_L) = H_{L'}.$$ 
\end{lemma}

\subsection{Hodge structures and K3 surfaces}
We recall some fundamental facts about the integral cohomology groups of K3 surfaces. Our main reference for this subsection is \cite{HuyK3}. We assume all our K3 surfaces to be algebraic.

For a K3 surface $Z$, the group $H^2(Z,\bZ)$ is a free abelian group of rank $22$. The cup-product on $H^2(Z,\bZ)$ is an integral, symmetric, non-degenerate, even bilinear form of signature $(3,19)$, thus $H^2(Z,\bZ)$ is isometric to the \textit{K3-lattice}:
\[
\Lambda_{K3}\coloneqq U^{\oplus3} \oplus E_8(-1)^{\oplus2}.
\]
Here, $U$ is the \textit{hyperbolic plane}, which is the unique even unimodular lattice of signature $(1,1)$, and $E_8$ is the unique even unimodular positive-definite lattice of rank 8. 

Additionally, $H^2(Z,\bZ)$ carries a Hodge structure of weight 2:
\[
H^2(Z,\bC) \simeq  H^{2,0}(Z) \oplus  H^{1,1}(Z)  \oplus
  H^{0,2}(Z),
\]
and we have $H^{2,0}(Z) \simeq H^0(Z,K_Z) \simeq \bC$. 
A group isomorphism $H^2(Z,\bZ) \simeq H^2(Z',\bZ)$ is called a \textit{Hodge isometry} if it is simultaneously an isometry and an isomorphism of Hodge structures. 
The \textit{N\'eron--Severi lattice} of $Z$ is $\NS(Z) \coloneqq H^{1,1}(Z) \cap H^2(Z,\bZ)$, seen as a sublattice of $H^2(Z,\bZ)$. Via the exponential sequence, we obtain an isomorphism $\Pic Z \simeq \NS(Z)$. 
The \textit{transcendental lattice} $T(Z)$ of $Z$ is defined to be the minimal, primitive, integral sub-Hodge structure of $H^2(Z,\bZ)$ such that $H^{2,0}(Z)\subset T(Z)_\bC$. Since $Z$ is algebraic, we have $T(Z) = \NS(Z)^\perp\subset H^2(Z,\bZ)$, see \cite[Lemma 3.3.1]{HuyK3}. The transcendental lattice is a primitive sublattice of $H^2(Z,\bZ)$, and it also carries a natural Hodge structure inherited from $H^2(Z,\bZ)$.

\begin{theorem}[{Torelli Theorem for K3 surfaces, \cite[Theorem 7.5.3]{HuyK3}}] \label{thm: torelli theorem}
    Let $Z$ and $Z'$ be K3 surfaces, and let $\psi\colon H^2(Z,\bZ) \simeq H^2(Z',\bZ)$ be a Hodge isometry. Then there exists an isomorphism $f\colon Z'\simeq Z$ such that $\psi = f^*$ if and only if $\psi$ maps an ample class on $Z$ to an ample class on $Z'$. 
\end{theorem}

\begin{theorem}[Derived Torelli Theorem for K3 surfaces \cite{Muk, OrlovK3}] \label{thm: derived torelli theorem}
    Let $Z$ and $Z'$ be K3 surfaces. Then $Z$ and $Z'$ are derived equivalent if and only if there exists a Hodge isometry $T(Z)\simeq T(Z')$.
\end{theorem}

Together, the Torelli Theorem and the Derived Torelli Theorem can be used to count Fourier--Mukai partners of K3 surfaces using the lattice theory of Nikulin \cite{Nik80}.

\begin{definition}
    For a K3 surface $Z$, we denote by $\FM(Z)$ the set of isomorphism classes of K3 surfaces $Z'$ such that $\Db(Z)\simeq \Db(Z')$. That is, $\FM(Z)$ is the set of Fourier--Mukai partners of $Z$.
\end{definition}

Note that $\FM(Z)$ is never empty, as we have $Z\in \FM(Z)$. 

\begin{theorem}[{Counting Formula for Fourier--Mukai partners of K3 surfaces, \cite[Theorem 2.3]{HLOY}}] \label{thm: counting formula} 
    Let $Z$ be a K3 surface. Then we have 
    \[
    |\FM(Z)| = \sum_{N\in \cG(\NS(Z))}|O(N) \backslash O(A_N) / O_{\mathrm{Hodge}}(T(Z))|.
    \]
    Here, $\cG(\NS(Z))$ is the genus of $\NS(Z)$, see Definition \ref{def: genus}, and $O_{\mathrm{Hodge}}(T(Z))$ denotes the group of Hodge isometries of $T(Z)$.
\end{theorem}

From Theorem \ref{thm: counting formula}, we see that the number $|\FM(Z)|$ depends only on $\NS(Z)$ and $O_{\mathrm{Hodge}}(T(Z))$, hence it is important to understand the group $O_{\mathrm{Hodge}}(T(Z))$. Note that $\pm\id_{T(Z)}$ are always Hodge isometries, so that we have an injective group homomorphism $\bZ/2\bZ\hookrightarrow O_{\mathrm{Hodge}}(T(Z))$. In fact, a K3 surface which is very general in a moduli space of lattice-polarised K3 surfaces of Picard rank $\rho\leq 19$ satisfies $O_{\mathrm{Hodge}}(T(Z)) \simeq \bZ/2\bZ$, see for example \cite[Lemma 3.9]{SZ}. 
Moreover, if $\rho$ is odd, then we always have $O_{\mathrm{Hodge}}(T(Z)) \simeq \bZ/2\bZ$:

\begin{lemma}[{\cite[Lemma 4.1]{Ogu}}]  \label{lem: odd picard rank few hodge isometries}
    Let $Z$ be a K3 surface whose Picard rank $\rho$ is odd. Then we have $O_{\mathrm{Hodge}}(T(Z)) \simeq \bZ/2\bZ$.
\end{lemma}

\section{The families and their associated K3 surfaces} \label{sec:families}

\subsection{Family 2-6(b)} \label{sec:family_2-6b_prelims}

The first family we study is Family 2-6(b) from Fanography \cite{fanography}, sometimes also referred to as \textit{special Verra threefolds}. We briefly recall the construction here. Consider a bidegree $(1,1)$-divisor $Y\subset \bP^2 \times \bP^2$ (in other words, $Y$ is a member of Family 2-32 in \cite{fanography}). Let $Z\in |-K_Y|$ be an anticanonical divisor. Then there exists a branched double cover 
\[
X \overset{2:1}{\longrightarrow} Y
\]
with branch locus $Z\subset Y$. The Fano threefold $X$ constructed in this way is a member of Family 2-6(b), and we will denote this deformation class by $\cX_{2\text{-}6(b)}$. We will refer to $Z$ as the K3 surface \textit{associated to} $X$. 
\begin{remark}
    In this section, we study special Verra threefolds. On the other hand, Fano threefolds in Family 2-6(a) are called \textit{ordinary Verra threefolds}. These are bidegree $(2,2)$ divisors in $\bP^2 \times \bP^2$, and will not be considered this section. 
    A categorical Torelli theorem for ordinary Verra threefolds has already been sketched in \cite[Remark 5.3]{GRZ}.
\end{remark}
The $(1,1)$-divisor $Y$ comes equipped with two projections $p_i\colon Y \to \bP^2$, and we denote 
\[
H_i \coloneqq p_i^*L \in \Pic Y,
\]
where $L\in \Pic\bP^2$ is the class of a line.
It follows from the Adjunction Formula that we have $$K_Y = -2(H_1 + H_2).$$ Finally, let $i_Z\colon Z \hookrightarrow Y$ be the inclusion map. Then we write 
\[
h_i \coloneqq i_Z^*H_i \in \Pic Z.
\]
We now show that the classes $h_1,h_2$ form a basis for $\Pic Z$, under the assumption that $Z$ is a very general member of $|-K_Y|$.
\begin{proposition} \label{prop: branch description 2-6(b)}
    Let $Z \in |-K_Y|$. 
    \begin{enumerate}
        \item If $Z$ is a general element of $|-K_Y|$, then it is a smooth K3 surface;
        \item If $Z$ is a very general element of $|-K_Y|$, then it is a smooth K3 surface and we have $\Pic Z = \NS(Z) = \bZ[h_1] \oplus \bZ[h_2]$ with Gram matrix 
        \begin{equation*}
        \begin{pmatrix}
            2 & 4 \\
            4 & 2
        \end{pmatrix} .
        \end{equation*}
    \end{enumerate}
\end{proposition}

\begin{proof}
    For (1), firstly, $Z$ is smooth by Bertini's Theorem. By adjunction and since $Z \in |-K_Y|$, we have $K_Z = 0$. Furthermore, $H^1(Z, \cO_Z) = 0$, hence $Z$ is a K3 surface.
    
    Note that we have $\Pic Y = \bZ [H_1] \oplus \bZ [H_2]$. Then $\Pic Z \simeq \Pic Y$ by \cite[Theorem 1]{RS}. For the intersections, recall that as a homology class, $[Z] = 2H_1 + 2H_2$. Then $h_1^2 = H_1^2 \cdot [Z] = H_1^2\cdot(2H_1+2H_2) = 2H_1^3 + 2H_1^2\cdot H_2 = 2(1)+0 = 2$. The rest of the intersections follow similarly. 
\end{proof}

\subsection{Family 2-8} \label{sec:family_2-8_prelims}

Next, we consider Family 2-8. Let $p\in \bP^3$ be a point, and denote the blow-up of $\bP^3$ in $p$ by $Y\coloneqq \Bl_p\bP^3$, and denote the blow-up morphism by 
\[
\pi \colon Y \to \bP^3.
\]
Let $Z\in |-K_Y|$ be an anticanonical divisor. Then there exists a branched double cover
\[
    X \overset{2:1}{\longrightarrow} Y
\]
with branch locus $Z\subset Y$. The Fano threefold $X$ is a member of Family 2-8, and we will denote the deformation class of these Fano threefolds by $\cX_{2\text{-}8}$.

We denote by $H \in \Pic Y$ the pullback of a plane in $\bP^3$ along $\pi$, and by $E\in \Pic Y$ the exceptional divisor of $\pi$. Since $Y$ is a blow-up, we have $\Pic Y = \bZ[H] \oplus \bZ[E]$.
The canonical divisor of $Y$ is 
\[
K_Y = 2E-4H.
\]
As before, let $i_Z\colon Z\hookrightarrow Y$ be the inclusion map. Then we denote
\[
    h \coloneqq i_Z^*H, \quad e \coloneqq i_Z^*E \in \Pic Z.
\]
We now show that the classes $h, e$ form a basis for $\Pic Z$, under the assumption that $Z$ is a very general member of $|-K_Y|$.

\begin{proposition} \label{prop: branch description 2-8}
    Let $Z \in |-K_Y|$. 
    \begin{enumerate}
        \item If $Z$ is a general element of $|-K_Y|$, then it is a smooth K3 surface;
        \item If $Z$ is a very general element of $|-K_Y|$, then it is a smooth K3 surface and we have $\Pic Z = \mathrm{NS}(Z) = \bZ [h] \oplus \bZ [e]$ where the Gram matrix for the lattice is 
        \begin{equation*}
            \begin{pmatrix}
                4 & 0 \\
                0 & -2
            \end{pmatrix}.
        \end{equation*}
    \end{enumerate}
\end{proposition}

\begin{proof}
    For (1), the proof is the same as the proof of Proposition \ref{prop: branch description 2-6(b)}(1).

    For (2), the fact that $\Pic Y \simeq \Pic Z$ follows from \cite[Theorem 3.8]{BG}. For the lattice, firstly note that $h^2 = 4$ and $e^2 = -2$. Indeed, as a homology class, $[Z] = 4H - 2E$. Thus, $h^2 = H \cdot H \cdot [Z] = H \cdot H \cdot (4H - 2E) = 4H^3 - 2(H^2 \cdot E) = 4(1) - 2(0) = 4$. Also, $e^2 = E \cdot E \cdot [Z] = E \cdot E \cdot (4H - 2E) = 4(E^2 \cdot H) - 2 E^3 = 4(0) - 2(1) = -2$. Furthermore, since $e$ corresponds to the exceptional divisor, we have $h \cdot e = 0$.
\end{proof}

\subsection{Family 3-1}

Consider Family 3-1 from \cite{fanography}. That is, $X$ is a double cover of \[Y = \bP^1 \times \bP^1 \times \bP^1\] branched in a divisor \[Z \in |-K_Y| = |\cO_Y(2,2,2)|.\] 
For $1\leq i \leq 3$, we write \[H_i \coloneqq p_i^*P,\] where $p_i\colon Y \to \bP^1$ denotes the projection onto the $i$-th factor, and $P\in \bP^1$ is a closed point. In other words, we have for example $\cO_Y(H_1) = \cO_Y(1,0,0)$.
Note that these three classes generate $\Pic Y$, i.e. we have \[\Pic Y = \bZ[H_1]\oplus \bZ[H_2]\oplus \bZ[H_3].\] 

Finally, for $1\leq i \leq 3$, let $h_i \coloneqq H_i|_Z$ be the pullbacks of $H_i$ to $Z$.

We show that the classes $h_1,h_2,h_3$ form a basis of $\Pic Z$, under the assumption that $Z$ is a very general member of $|-K_Y|$.

\begin{proposition}\label{prop: branch description 3-1}
     Let $Z \in |-K_Y|$. 
    \begin{enumerate}
        \item If $Z$ is a general element of $|-K_Y|$, then it is a smooth K3 surface;
        \item If $Z$ is a very general element of $|-K_Y|$, then it is a smooth K3 surface and we have $\Pic Z = \NS(Z) = \bZ[h_1] \oplus \bZ[h_2] \oplus \bZ[h_3]$ where the Gram matrix for the lattice is
        \begin{equation*}
        \begin{pmatrix}
            0 & 2 & 2 \\
            2 & 0 & 2 \\
            2 & 2 & 0
        \end{pmatrix} .
        \end{equation*}
    \end{enumerate}
\end{proposition}

\begin{proof}
    For (1), the proof is again the same as the proof of Proposition \ref{prop: branch description 2-6(b)}(1). 

    For (2), the fact that $\Pic Y \simeq \Pic Z$ again follows from \cite[Theorem 3.8]{BG}. For the lattice, firstly note that $H_i^2 = 0$ on each copy of $\bP^1$, hence $h_i^2 = 0$ for $i=1,2,3$. Now let $i \neq j$. As a homology class, $[Z] = 2H_1 + 2H_2 + 2H_3$. Then $h_i \cdot h_j = H_i \cdot H_j \cdot (2H_1 + 2H_2 + 2H_3) = 2$, since $H_i \cdot H_j \cdot H_k = 1$ for $k \in \{ 1,2,3 \} \setminus \{ i, j \}$, and the triple product vanishes for the other choices of $k$. 
\end{proof}

\subsection{Verra fourfolds}

Finally, we consider Verra fourfolds. Let $Y = \bP^2 \times \bP^2$ and let $V_3 \in |\cO_Y(2,2)|$. Then $V_3$ is called an \emph{ordinary Verra threefold}. There exists a branched double cover 
\begin{equation*}
    V_4 \to Y
\end{equation*}
with branch locus $V_3 \subset Y$. The Fano fourfold $V_4$ is called a \emph{Verra fourfold}. 

We note that Verra fourfolds distinguish themselves from the other three families of Fano double covers under consideration in two ways. Firstly, Verra fourfolds have dimension 4, whereas members of the other three families are threefolds. Secondly, the branch locus of the double cover $V_4 \to Y$ is not an anticanonical divisor of $Y$, but rather it is itself a Fano threefold.

For an ordinary Verra threefold $V_3 \subset \bP^2 \times \bP^2$, we have natural projections $p_i\colon V_3 \to \bP^2$, where $i = 1,2$. We write $H_i \coloneqq p_i^*L$, where $L \in \Pic \bP^2$ is the class of a line.  

\begin{proposition} \label{prop: branch description V4}
    Let $V_3 \in |\cO_Y(2,2)|$.
    \begin{enumerate}
        \item[(1)] If $V_3$ is a general element of $|\cO_Y(2,2)|$, then $V_3$ is a smooth Fano threefold;
        \item[(2)] If $V_3$ is a very general element of $|\cO_Y(2,2)|$, then it is a smooth Fano threefold and we have $\Pic V_3 = \NS(V_3) = \bZ[H_1]\oplus \bZ[H_2]$. Moreover, the intersection form of $\NS(V_3)$ is given by
        \[
        (aH_1 + bH_2)^3 = 6ab(a+b).
        \]
    \end{enumerate}
\end{proposition}
\begin{proof}
    As for the other three families, the first claim follows from Bertini's Theorem, and the fact that $V_3$ is a Fano threefold follows from the adjunction formula.

    Note that since $[V_3] = 2H_1+2H_2 \in \Pic Y$, it follows that we have 
    \begin{align*}
         (aH_1+bH_2)^3 &= a^3H_1^3\cdot[V_3] + 3a^2b H_1^2H_2\cdot [V_3] + 3ab^2 H_1H_2^2\cdot [V_3] + b^3H_2^3\cdot[V_3] \\
         &= a^3\cdot0 + 3a^2b\cdot 2 + 3ab^2 \cdot 2 + b^3 \cdot 0 \\&= 6ab(a+b),
    \end{align*}
    as required.
\end{proof}

The inclusion $V_3 \hookrightarrow Y = \bP^2 \times \bP^2$ followed by a projection to one of the copies of $\bP^2$ gives two conic bundles $p_i \colon V_3 \to \bP^2$; one for each projection $i=1,2$. Similarly, the restriction of the double cover map $V_4 \to Y = \bP^2 \times \bP^2$ to each of the copies of $\bP^2$ gives two quadric surface bundles $\pi_i \colon V_4 \to \bP^2$; one for each restriction $i=1,2$.

\subsection{$Z$-generality}
From now on, whenever we consider a Fano threefold $X$ contained in one of the families $\mathcal{X}_{2\text{-}6(b)}, \mathcal{X}_{2\text{-}8},$ and $\mathcal{X}_{3\text{-}1}$, we shall usually require the ramification locus to satisfy one of the conclusions of Proposition \ref{prop: branch description 2-6(b)}, Proposition \ref{prop: branch description 2-8}, and Proposition \ref{prop: branch description 3-1}, respectively. 
Similarly, for Fano fourfolds, we shall usually require the branch ordinary Verra threefold to be a smooth Fano variety of minimal Picard rank.
\begin{definition} \leavevmode
    \begin{enumerate}
    \item We denote by $\Lts$ the lattice of rank $2$ with Gram matrix given by 
    \begin{equation}\label{eq: Lts}
    \begin{pmatrix}
        2 & 4 \\ 4 & 2
    \end{pmatrix}.
    \end{equation}
    \item We denote by $\Lte$ the lattice of rank $2$ with Gram matrix given by 
    \[
    \begin{pmatrix}
         4 & 0 \\ 0 & -2
    \end{pmatrix}.
    \]
    \item We denote by $\Lto$ the lattice of rank $2$ with Gram matrix given by 
    \[
    \begin{pmatrix}
         0 & 2 & 2\\ 2 & 0 & 2 \\ 2 & 2 & 0
    \end{pmatrix}.
    \]    
    \end{enumerate}
\end{definition}
As we saw in Proposition \ref{prop: branch description 2-6(b)}, Proposition \ref{prop: branch description 2-8}, and Proposition \ref{prop: branch description 3-1}, the N\'eron--Severi lattice of the branch divisor $Z$ of a very general Fano threefold in one of the families $\mathcal{X}_{2\text{-}6(b)},\mathcal{X}_{2\text{-}8}$, and $\mathcal{X}_{3\text{-}1}$ is isometric to $\Lts,\Lte,$ and $\Lto$, respectively.

\begin{definition} \label{def: Z general}
\leavevmode
    \begin{enumerate}
        \item A Fano threefold $X$ in Family 2-6(b), Family 2-8, or Family 3-1 is called \textit{K3-general} if the branch locus $Z\subset Y$ of the double cover $X\to Y$ is a smooth K3 surface.
        \item A Fano threefold in Family 2-6(b), Family 2-8, or Family 3-1 is called \textit{$Z$-general} if it is K3-general and the N\'eron--Severi lattice of its associated K3 surface is isometric to $\Lts, \Lte$, or $\Lto$, respectively.
        \item A Verra fourfold is called \textit{Fano-general} if its branch locus is a smooth ordinary Verra threefold.
    \end{enumerate}
\end{definition}

\begin{remark}\label{rem: Z general K3 general}
    By Proposition \ref{prop: branch description 2-6(b)}, Proposition \ref{prop: branch description 2-8}, and Proposition \ref{prop: branch description 3-1}, a general member of one of the three families of Fano threefolds under consideration is K3-general, and a very general member is $Z$-general. Similarly, a general Verra fourfold is Fano-general.
\end{remark}

\begin{remark}
    As we will see in the following sections, ordinary Verra threefolds are standard conic bundles over $\bP^2$. Here, we say that a conic bundle $f\colon V\to \bP^2$ is \emph{standard} if $\Pic V \simeq f^*\Pic\bP^2 \oplus \bZ$. In particular, all smooth ordinary Verra threefolds have Picard rank 2, as noted in \cite{Ver}.
\end{remark}

\begin{remark} \label{rmk: why not 2-8(b)}
    Family 2-8 splits into two subfamilies called Family 2-8(a) and Family 2-8(b) on Fanography \cite{fanography}. We say that $X$ is a member of Family 2-8(a) if the scheme-theoretic intersection $Z\cap E$ is smooth, and $X$ is a member of Family 2-8(b) if $Z\cap E$ is singular but reduced. In this paper, we will mostly consider the former. More precisely, we prove a categorical Torelli theorem for $Z$-general Fano threefolds in Family 2-8. A $Z$-general Fano threefold in Family 2-8 is always a member of Family 2-8(a). Therefore, the categorical Torelli problem for very general members of Family 2-8(b) is still open.
\end{remark}

\section{The Fano is determined by the branch} \label{sec:fano is determined by branch}

In this section, we prove the following result.

\begin{theorem} \label{thm: branch determines Fano}
    Let $X$ be a $Z$-general Fano threefold contained in one of the families  $\mathcal{X}_{2\text{-}6(b)}, \mathcal{X}_{2\text{-}8},$ and $\mathcal{X}_{3\text{-}1}$, and let $X'$ be a Fano threefold contained in the same family. Let $Z$ and $Z'$ be the branch loci of $X$ and $X'$, respectively. Then, if there is an isomorphism of the branch divisors, the Fano threefolds $X$ and $X'$ are isomorphic:
    \[
    Z\simeq Z' \implies X \simeq X'.
    \]
    Moreover, suppose $V_4$ and $V_4'$ are Verra fourfolds. Let $V_3$ and $V_3'$ be the branch loci of $V_4$ and $V_4'$, respectively. Then, if there is an isomorphism $V_3 \simeq V_3'$, the Verra fourfolds $V_4$ and $V_4'$ are isomorphic:
    \[
    V_3 \simeq V_3'\implies V_4 \simeq V_4'.
    \]
\end{theorem}
\begin{proof}
    Combine Proposition \ref{prop: branch determines cover 2-6}, Proposition \ref{prop: branch determines cover 2-8}, Proposition \ref{prop: branch determines cover 3-1}, and Proposition \ref{prop: branch determines cover V4} below.
\end{proof}

\begin{remark}
    Suppose $X$ and $X'$ are Fano threefolds in one of the three families under consideration. If $X$ and $X'$ have isomorphic branch loci, then $X$ is K3-general (resp. $Z$-general) if and only if $X'$ is. This follows immediately from the definition of K3-generality (resp. $Z$-generality), as it is a property of the branch divisor. Similarly, for two Verra fourfolds $V_4$, $V_4'$ with isomorphic branch loci $V_3$ and $V_3'$, respectively, we have that $V_4$ is Fano-general if and only if $V_4'$ is.
\end{remark}

\begin{remark}
    The first part of Theorem \ref{thm: branch determines Fano} can be equivalently expressed as follows: 
    Let $Y$ be a bidegree $(1,1)$ divisor in $\bP^2 \times \bP^2$, the blow-up $\Bl_p(\bP^3)$, or $\bP^1 \times \bP^1 \times \bP^1$. Then the rational period map
    \begin{equation*}
        |-K_Y|/\Aut(Y) \dashrightarrow \mathcal{M}_{L}
    \end{equation*}
    is generically injective, where $\mathcal{M}_{L}$ denotes the moduli space of $L$-polarised K3 surfaces. Here $L$ is $\Lts$, $\Lte$, or $\Lto$, respectively.
\end{remark}

\subsection{Family 2-6(b)} We first prove Theorem \ref{thm: branch determines Fano} for special Verra threefolds, i.e. for Family 2-6(b).

\begin{proposition} \label{prop: branch determines cover 2-6}
    Let $X,X' \in \mathcal{X}_{2\text{-}6(b)}$ be $Z$-general. Then, if there is an isomorphism of branch divisors $Z \simeq Z'$, we have $X\simeq X'$. 
\end{proposition}

Before we prove Proposition \ref{prop: branch determines cover 2-6}, let us make some remarks on the geometry of the situation. 
\begin{definition} \leavevmode
    \begin{enumerate}
        \item A \textit{sextic structure} on a K3 surface $Z$ is a double cover $f\colon Z\to \bP^2$ branched in a sextic curve $C\subset \bP^2$. 
        \item If  $f\colon Z\to \bP^2$ is a sextic structure, then we call $h\coloneqq f^*\mathcal{O}(1)\in \NS(Z)$ the \textit{ample divisor of the sextic structure $f$.}
        \item Two sextic structures $f\colon Z \to \bP^2$ and $f'\colon Z'\to \bP^2$, with branch loci $C$ and $C'$, respectively, are said to be \emph{isomorphic} if there is an automorphism $\phi\colon \bP^2 \to \bP^2$ satisfying $\phi(C) = C'$. Equivalently, there exists an isomorphism $Z\simeq Z'$ taking the ample divisor of $f$ to the ample divisor of $f'$.
        \item The sextic structure on $Z$ whose ample divisor is $h\in \NS(Z)$ is denoted $f_h$.
\end{enumerate}
\end{definition}

Recall that a Fano threefold in Family 2-6(b) is a double cover of a bidegree $(1,1)$ divisor $Y$ in $\bP^2 \times \bP^2$. Let $Z\in |-K_Y|$ be a K3 surface, and denote by  $p_i\colon Y \to \bP^2$ the projection onto the $i$-th factor. Then the composition 
\[
Z\overset{i_Z}{\hookrightarrow} Y \overset{p_i}{\to} \bP^2
\]
is a sextic structure on $Z$, whose ample divisor is $h_i$, see Proposition \ref{prop: branch description 2-6(b)}. Here, $i_Z$ denotes the inclusion map. Our first goal is to prove that $Z$ admits precisely two sextic structures up to isomorphism, see Lemma \ref{lem: two sextic structures} below. 
\begin{lemma}\label{lem: no -2 curves}
    The lattice $\Lts$ contains no classes of square $-2$. As a consequence, for a K3 surface $Z$ with $\NS(Z)\simeq \Lts$, any class $D\in \NS(Z)$ which is effective and satisfies $D^2>0$ is ample, and for any class $h\in \NS(Z)$ with $h^2>0$, either $h$ or $-h$ is effective.
\end{lemma}
\begin{proof}
    For the sake of contradiction, suppose that we have integers $a,b\in \bZ$ such that 
    \begin{equation}\label{eq: impossible -2 class}
            (ah_1+bh_2)^2 = 2a^2 + 8ab + 2b^2 = -2.
    \end{equation}
    Dividing  \eqref{eq: impossible -2 class} by $2$ and reducing modulo $4$, we find $a^2 + b^2 \equiv -1 \pmod 4$. However, the only squares modulo $4$ are $0$ and $1$, so this is a contradiction. 
    The final claim follows for example from \cite[Corollary 8.1.7]{HuyK3}.
\end{proof}

\begin{lemma}\label{lem: find the correct basis for lts}
    Let $Z$ be a K3 surface with $\NS(Z)\simeq \Lts$.
    Let $h\in \NS(Z)$ be an effective class of square $h^2 = 2$, hence $h$ is ample by Lemma \ref{lem: no -2 curves}.
    Then there exists another ample class $h'\in \NS(Z)$ such that the Gram matrix of $\NS(Z)$ with respect to the basis $h,h'$ is 
    \[
    \begin{pmatrix}
        2 & 4 \\ 4 & 2
    \end{pmatrix}.
    \]
    In particular, there is an isometry $\sigma \in O(\NS(Z))$ such that $\sigma(h) = h_1$ and $\sigma(h') = h_2$.
\end{lemma}
\begin{proof}
    Let $D\in \NS(Z)$ be any class such that $h,D$ generate $\NS(Z)$. Such a class exists because $h$ is primitive (as any class of square $2$ must be primitive). We wish to find a $k\in \bZ$ such that $(D+kh)^2 = 2$, that is, 
    \begin{equation}\label{eq: quadratic equation for middle schoolers}
        \frac{1}{2}(D+kh)^2 = \frac{1}{2}D^2 + kD\cdot h + k^2 = 1.
    \end{equation} This is a quadratic equation in $k$ whose discriminant is $(D\cdot h)^2 - 2D^2 + 4 = 16$. Here, we use the fact that $h,D$ generate $\NS(Z)$, hence the determinant of their Gram matrix must be $2D^2 - (D\cdot h)^2 = -12 = \disc(\Lts)$. In particular, there exist integer solutions to \eqref{eq: quadratic equation for middle schoolers}, hence we may find $h' = D+kh$ such that $(h')^2 = 2.$ Possibly replacing $h'$ by $-h'$, we may assume $h'$ is effective by Lemma \ref{lem: no -2 curves}. This implies that $h\cdot h' \geq 0$, since $\NS(Z)$ does not contain any $(-2)$-curves. Since the determinant of the Gram matrix of the basis $h,h'$ must be $-12$, we see $h\cdot h' = 4$.
\end{proof}

\begin{lemma}\label{lem: two sextic structures}
    Let $Z\in |-K_Y|$ be a smooth K3 surface with $\NS(Z) = \langle h_1, h_2\rangle \simeq \Lts$. Then each sextic structure on $Z$ is isomorphic to either $f_{h_1}$ or $f_{h_2}$.
\end{lemma}
\begin{proof}   
    Let $h\in \NS(Z)$ be an ample divisor with $h^2 = 2$, and let $h'\in \NS(Z)$ be the ample divisor obtained from Lemma \ref{lem: find the correct basis for lts}. 
    We claim that $$\left\{\frac{1}{2}h,\frac{1}{2}h'\right\} = \left\{\frac{1}{2}h_1,\frac{1}{2}h_2\right\} \subset A_{\NS(Z)}.$$ 
    Indeed, it follows from a straightforward computation that we have 
    \begin{equation}\label{eq: discriminant group Ats}
        A_{\NS(Z)} = \left\langle \frac{1}{2}h_1 \right\rangle \oplus \left\langle \frac{1}{3}h_1 -\frac{1}{6}h_2 \right \rangle \simeq \bZ/2\bZ \oplus \bZ/6 \bZ.  
    \end{equation}
    Hence $A_{\NS(Z)}$ contains precisely three elements of order $2$, namely $$\frac{1}{2}h_1, \quad \frac{1}{2}h_2, \quad \frac{1}{2}(h_1+h_2).$$
    One easily checks that $$q\left(\frac{1}{2}h_1\right) = q\left(\frac{1}{2}h_2\right) \equiv \frac{1}{2} \pmod{2\bZ}, \quad q\left(\frac{1}{2}(h_1 + h_2)\right) \equiv 1\pmod{2\bZ}.$$ Since $q(\frac{1}{2}h)\equiv \frac{1}{2}\pmod{2\bZ}$, we must have $\frac{1}{2}h\in \left\{\frac{1}{2}h_1,\frac{1}{2}h_2\right\}$, as claimed. We assume $\frac{1}{2}h = \frac{1}{2}h_1$ and $\frac{1}{2}h' = \frac{1}{2}h_2$, as the other case follows by symmetry. 

    Let $\sigma\colon \NS(Z) \simeq \NS(Z)$ be an isometry such that $\sigma(h) = h_1$ and $\sigma(h') = h_2$. Such an isometry exists by Lemma \ref{lem: find the correct basis for lts}. Note that $\sigma$ preserves the ample cone, since $h$ and $h_1$ are both ample.
    We claim that $\overline{\sigma} = \pm \id_{A_{\NS(Z)}}$. Indeed, $\overline{\sigma}$ fixes the two classes $\frac{1}{2}h_1$ and $\frac{1}{2}h_2$. 
    Since $\overline\sigma$ fixes $\frac{1}{2}h_1$, we may write $\overline\sigma  = (\sigma_1, \sigma_2)$, where $\sigma_1 = \id_{\left\langle \frac{1}{2}h_1 \right\rangle}$ and $\sigma_2$ is an isometry of $\left\langle \frac{1}{3}h_1 -\frac{1}{6}h_2 \right \rangle.$ However, since $\left\langle \frac{1}{3}h_1 -\frac{1}{6}h_2 \right \rangle \simeq \bZ/6 \bZ$, it follows that we must have $\sigma_2 = \pm \id,$ so that $\overline\sigma = \pm \id_{A_{\NS(Z)}}$, as required.
    As a consequence, $\sigma$ can be extended to a Hodge isometry $\Psi\colon H^2(Z,\bZ)\simeq H^2(Z,\bZ)$ by Lemma \ref{lem: extending isometries}. Since $\sigma$ preserves the ample cone, $\Psi = \psi^*$ for some automorphism $\psi\in \Aut(Z)$, and this automorphism by construction satisfies $\psi^*h = h_1$, hence $f_h \simeq f_{h_1}$.
    \end{proof}

\begin{remark}\label{rmk: at most two, but usually two}
    As we can see in the proof of Lemma \ref{lem: two sextic structures}, a K3 surface $Z$ with $\NS(Z)\simeq \Lts$ has at most two sextic structures up to isomorphism.  Moreover, the two sextic structures are isomorphic if and only if there exists a Hodge isometry $\Psi\in O_{\mathrm{Hodge}}(T(X))$ such that $\overline\Psi = \overline\tau \in O(A_{\NS(Z)})$, where $\tau$ is the isometry of $\NS(Z)$ which swaps $h_1$ and $h_2$. In the very general case, we have $O_{\mathrm{Hodge}}(T(Z)) \simeq \bZ/2\bZ$, and such a Hodge isometry does not exist, but in special cases there exist more Hodge isometries of $T(Z)$. Hence, there may exist $Z$-general Fano threefolds in Family 2-6(b) whose associated K3 surfaces have only one sextic structure up to isomorphism.
\end{remark}

\begin{lemma}\label{lem: bidegree (1,1) divisors are projective bundles over the plane}
    Let $Y$ be a bidegree $(1,1)$ divisor in $\bP^2\times \bP^2$, and let $p_i\colon Y \to \bP^2$ denote the projection onto the $i$-th factor, where $i = 1,2$.
    Then for each $i=1,2$ there is an isomorphism $Y\simeq \bP(T_{\bP^2})$ fitting into a commutative diagram
    \[
    \xymatrix{
    Y \ar[rr]^\simeq \ar[dr]_{p_i}&& \bP(T_{\bP^2}) \ar[dl]^{\pi} \\
    & \bP^2 &
    }
    \] 
    where $\pi\colon \bP(T_{\bP^2}) \to \bP^2$ denotes the canonical projection, and $T_{\bP^2}$ is the tangent bundle of $\bP^2$.
\end{lemma}
\begin{proof}
    See \cite[Lemma 5.7]{BFConn}.
\end{proof}

Recall that for an anticanonical K3 surface $Z\in |-K_Y|$, we write $i_Z\colon Z\hookrightarrow Y$ for the inclusion map.
\begin{lemma}\label{lem: isomorphism of sextic structures implies isomorphic fanos}
    Let $Y\coloneqq \bP(T_{\bP^2})$, and let $Z, Z'\in |-K_Y|$ smooth K3 surfaces with $\NS(Z) \simeq \Lts \simeq \NS(Z')$. Let $f \coloneqq \pi \circ i_Z\colon Z \to \bP^2$ and $g\coloneqq \pi \circ i_{Z'}\colon Z'\to \bP^2$ be the induced sextic structures on $Z$ and $Z'$. Assume moreover that $f$ and $g$ are isomorphic via an automorphism $\phi\colon \bP^2 \simeq \bP^2$. Then $\phi$ lifts to an automorphism $\tilde{\phi} \colon Y \simeq Y$ such that $\tilde{\phi}(Z) = Z'$. 
\end{lemma}
\begin{proof}
    By assumption, we have an isomorphism $\psi\colon Z\simeq Z'$ fitting into a commutative diagram:
    \[
    \xymatrix{
    Z \ar@{^(->}[d]_{i_Z} \ar[r]^\simeq_\psi & Z' \ar@{^(->}[d]^{i_{Z'}} \\
    Y \ar[d]_\pi & Y \ar[d]^{\pi}\\
    \bP^2 \ar[r]^\simeq_\phi & \bP^2,
    }
    \]
    where the compositions of the vertical maps are the sextic structures $f$ and $g$. Since any automorphism of $\bP^2$ lifts to an automorphism of $Y$, it follows that there exists an automorphism $\tilde{\phi}\in \Aut(Y)$ such that the diagram 
    \[
    \xymatrix{
    Y \ar[r]^\simeq_{\tilde{\phi}} \ar[d] & Y \ar[d] \\
    \bP^2 \ar[r]^\simeq_{\phi}& \bP^2
    }
    \]
    commutes. 

    Therefore, $Z'' \coloneqq (\tilde{\phi}\circ i_Z)(Z)$ and $Z'$ are two K3 surfaces in $Y$ such that the double covers $Z'\to \bP^2$ and $Z''\to \bP^2$ have precisely the same branch locus $C \subset \bP^2$. Finally, we may use Lemma \ref{lem: branch determines anticanonical divisor} below, combined with Lemma \ref{lem: bidegree (1,1) divisors are projective bundles over the plane}, to conclude that $Z'' = Z'$, as required.
\end{proof}

\begin{lemma}\label{lem: branch determines anticanonical divisor}
    Let $Y$ be a bidegree $(1,1)$ divisor in $\bP^2\times \bP^2$. Let $Z, Z'\in |-K_Y|$ be K3 surfaces with $\NS(Z) \simeq \Lts \simeq \NS(Z')$. Let $C$ and $C'$ be the discriminant sextics of $f_{h_1}$ and $f_{h'_1}$, respectively. Then, if $C = C'$ as subvarieties of $\bP^2$, we have $Z = Z'$ as subvarieties of $Y$.
\end{lemma}
\begin{proof}
    Since $C = C'$, we find that $Z$ and $Z'$ are abstractly isomorphic, as they are both double covers of $\bP^2$ with the same branch locus. Fix an isomorphism $\psi\colon Z\simeq Z'$ fitting in a commutative triangle
    \begin{equation*}
    \xymatrix{
    Z \ar[dr]_{f_{h_1}} \ar[rr]^\simeq_\psi && Z' \ar[dl]^{f_{h_1'}} \\
    &\bP^2. &
    }
    \end{equation*}
    Then $\psi$ also induces an isomorphism between $f_{h_2}$ and $f_{h_2'}$:
    \begin{equation*}
    \xymatrix{
    Z \ar[d]_{f_{h_2}} \ar[r]^\simeq_\psi & Z' \ar[d]^{f_{h_2'}} \\
    \bP^2 \ar[r]_\phi^\simeq&\bP^2. 
    }
    \end{equation*}
    We claim that $\phi = \id_{\bP^2}$. Note that this immediately implies the result, as the embedding $Z \hookrightarrow \bP^2 \times \bP^2$ is the product of the two sextic structures $f_{h_1}$ and $f_{h_2}$. 
    
    Since $(\id_{\bP^2}\times \phi)(Z) = Z'$, we have 
    \begin{equation}\label{eq: K3 in hyperplane section}
    Z'\subset S \coloneqq Y\cap \left((\id_{\bP^2}\times \phi)(Y)\right).
    \end{equation}
    Assume for the sake of contradiction that $\phi\neq \id_{\bP^2}$. Then $S$ is the intersection of $Y$ with a different $(1,1)$ divisor in $\bP^2 \times \bP^2$, thus $S$ is a divisor in $Y$. Moreover, we have $S = H_1 + H_2 \in \Pic Y$. Moreover, by \eqref{eq: K3 in hyperplane section}, we can write $S = Z' + T$ for some effective divisor $T$. However, since $Z'\in |-K_Y|$, we have $Z' = 2H_1 + 2H_2$, thus $T = S - Z = -H_1 - H_2$. Therefore $T$ is not an effective divisor, which is a contradiction.  Therefore, we have $\phi = \id_{\bP^2}$, and the result follows.
\end{proof}

\begin{proof}[Proof of Proposition \ref{prop: branch determines cover 2-6}]
    By Lemma \ref{lem: two sextic structures}, there exists an isomorphism $\psi\colon Z\simeq Z'$, for which the ample divisor $\psi^*h_1'$ is equal to either $h_1$ or $h_2$. Suppose that we have $\psi^*h_1'= h_1$, as the other case follows by symmetry.
    Then, by Lemma \ref{lem: isomorphism of sextic structures implies isomorphic fanos}, there is an automorphism $\tilde\phi\colon Y\simeq Y$ satisfying $\tilde\phi(Z) = Z'$, hence we have $X\simeq X'$.
\end{proof}

\subsection{Family 2-8}
In this subsection, we prove the following proposition.
\begin{proposition} \label{prop: branch determines cover 2-8}
    Let $X, X' \in \cX_{2\text{-}8}$ be $Z$-general. If there is an isomorphism of branch divisors $Z \simeq Z'$, then $X \simeq X'$.
\end{proposition}

Before we proceed to the proof of Proposition \ref{prop: branch determines cover 2-8}, we first study the embeddings of the branch divisor $Z$ into $Y = \Bl_p\bP^3$.

\begin{remark} \label{rem: 2-8 singular quartic}
    The surface $Z$ is the strict transform of a quartic hypersurface in $\bP^3$ with an ordinary double point at $p$. 
    
    Indeed, let $\pi \colon Y \to \bP^3$ be the blow-up map, and $E$ the exceptional divisor. Denote $H_Y \coloneqq \pi^* \cO_{\bP^3}(1)$. We have $K_Y = \pi^* K_{\bP^3} + 2E = -4H + 2E$. Hence $-K_Y = 4H - 2E$. For such a blow-up, we have $\pi_* \cO_Y(-kE) = I_p^k$ for $k \geq 0$. We have $\cO_Y(-K_Y) = \cO_Y(4H - 2E) = \pi^* \cO_{\bP^3}(4) \otimes \cO_Y(-2E)$. Applying the projection formula, we get $\pi_* \cO_Y(-K_Y) = \cO_{\bP^3}(4) \otimes \pi_* \cO_Y(-2E) = I_p^2(4)$. This means that the global sections of $\cO_Y(-K_Y)$ are precisely degree-$4$ polynomials on $\bP^3$ vanishing to order at least $2$ at the point $p$, in other words, we have a double point at $p$. Taking the strict transform removes the multiplicity-$2$ component along $E$, giving a divisor in $|-K_Y|$.
\end{remark}

\begin{lemma}\label{lem: ample divisor dictates aut-orbit of embeddings 2-8}
    Let $Z\in |-K_Y|$ be a general anticanonical divisor of $Y = \Bl_p\bP^3$, so that $Z$ is a smooth K3 surface.
    Let $i\colon Z \hookrightarrow Y$ and $j\colon Z\hookrightarrow Y$ be two embeddings. Assume that $i^*(H) = j^*(H) = h \in \NS(Z)$. Then there is an automorphism $\psi\colon Y\simeq Y$ making the following diagram commute:
    \begin{equation}\label{eq: please commute}
        \xymatrix{
        & Z \ar@{_(->}[dl]_i \ar@{^(->}[dr]^j&\\
        Y \ar[rr]^\simeq_\psi && Y
        }
    \end{equation}
\end{lemma}
\begin{proof}
    Consider the morphism $\phi \coloneqq \phi_{|h|}\colon Z \to \bP^3$. The image $\phi(Z)$ is a singular quartic K3 surface with an ordinary double point in $p$. Explicitly, the map $\phi$ is the restriction of the blow-up $\Bl_p\colon Y \to \bP^3$, and the exceptional $(-2)$-curve of $Z$ is contracted to an ordinary double point (see Remark \ref{rem: 2-8 singular quartic}). 

    By the assumption that $i^*H = j^*H = h$, there exists an automorphism $\psi'\colon \bP^3 \simeq \bP^3$ such that the diagram 
    \[
    \xymatrix{
        & Z \ar@{_(->}[dl]_{i} \ar@{^(->}[dr]^j & \\
        Y \ar[d] && Y \ar[d] \\
        \bP^3 \ar[rr]^\simeq_{\psi'} && \bP^3
        }
    \]
    commutes. In particular, $\psi'$ fixes the point $p\in \bP^3$, hence it lifts to the desired automorphism $\psi\in \Aut(Y)$ making \eqref{eq: please commute} commute.
\end{proof}

\begin{proof}[Proof of Proposition \ref{prop: branch determines cover 2-8}]
    By $Z$-generality of $X$, the Néron--Severi lattice $\mathrm{NS}(Z) = \Pic Z$ is given by the Gram matrix 
    \begin{equation*}
            \begin{pmatrix}
        4 & 0 \\
        0 & -2
    \end{pmatrix}. 
    \end{equation*} 
    Suppose we have an isomorphism of branch divisors $\phi \colon Z \simeq Z'$. In this setting, $\NS(Z)$ is generated by $h \coloneqq H|_Z$ and $e \coloneqq E|_Z$, and $\NS(Z')$ is similarly generated by $h'$ and $e'$ by Proposition \ref{prop: branch description 2-8}. The isomorphism $\phi$ induces an isometry $\phi^* \colon \NS(Z) \simeq \NS(Z')$. We claim we must have $\phi^*(h) = h'$ and $\phi^*(e) = e'$.

    Indeed, let $a,b\in \bZ$ such that  $\phi^*(h) = ah' + be'$. Since $\phi$ is an isomorphism, $\phi^*(h)$ must be effective, since $h$ is effective. It follows that $\phi(h)\cdot e' = -2b^2\geq 0$, hence $b = 0$. From this, it easily follows that $a = \pm 1$. However, since $-h'$ is not effective, we obtain $\phi(h) = h'$. 
    Now, since $e$ is orthogonal to $h$, it follows that $\phi(e)$ is orthogonal to $\phi(h) = h'$. We have $h'^\perp = \langle e' \rangle \subset \NS(Z')$, so we obtain $\phi(e) = \pm e'$. Finally, since $e$ is effective, $\phi(e)$ is also effective, and we must have $\phi(e) = e'$. 

    Now, write $i_Z\colon Z\hookrightarrow Y$ and $i_{Z'}\colon Z'\hookrightarrow Y$ for the inclusions. Then $i_Z$ and the composition $i_{Z'} \circ \phi$ are two different embeddings of $Z$ into $Y$, and by the above computations we have $$i_Z^*(H)= h = \phi^*(h') = (i_{Z'}\circ \phi)^*(H).$$
    This means that we may use Lemma \ref{lem: ample divisor dictates aut-orbit of embeddings 2-8} to obtain an automorphism $\psi\in \Aut(Y)$ fitting into a commutative diagram
    \begin{equation*}
        \xymatrix{
        Z \ar[r]^\simeq_\phi \ar@{^(->}[d]_{i_{Z}} & Z' \ar@{^(->}[d]^{i_{Z'}} \\
        Y \ar[r]^\simeq_\psi & Y.
        }
    \end{equation*}
    From this, we immediately obtain an isomorphism of the double covers $X\simeq X'$, as required.
\end{proof}

\subsection{Family 3-1} Next, we consider Theorem \ref{thm: branch determines Fano} for Family 3-1.
\begin{proposition} \label{prop: branch determines cover 3-1}
    Let $X, X' \in \cX_{3\text{-}1}$ be $Z$-general. If there is an isomorphism of branch divisors $Z \simeq Z'$, then $X \simeq X'$.
\end{proposition}

\begin{definition}\label{def: elliptic fibrations} \leavevmode
    \begin{enumerate}
        \item An \textit{elliptic fibration} of a K3 surface $Z$ is a surjective morphism $f\colon Z\to \bP^1$.
        \item Let $f\colon Z\to \bP^1$ be an elliptic fibration, and let $x\in \bP^1$ be a closed point. Then the class $[Z_x] \in \NS(Z)$ is called the \textit{fibre class} of $f$. The fibre class is independent of the chosen point $x$.
        \item For an elliptic fibration $f\colon Z\to \bP^1$ with fibre class $F\in \NS(Z)$, a curve $C\subset Z$ with the property that $C\cdot F > 0$ is called a \textit{multisection} of $f$. 
        \item The \textit{multisection index} $t\geq 1$ of an elliptic fibration $f\colon Z\to \bP^1$ is the minimal degree of a multisection. If $F\in \NS(Z)$ is the fibre class of $f$, then $t$ is the divisibility of $F$ in $\NS(Z)$ (see \eqref{eq: divisibility}): 
        \[
        t = \mathrm{div}(F).
        \]
        \item Two elliptic fibrations $f\colon Z\to \bP^1$, $f'\colon Z'\to \bP^1$ are said to be \textit{isomorphic} if there is a commutative square as below, where the horizontal maps are isomorphisms
        \begin{equation*}
            \xymatrix{
            Z \ar[r]^\simeq \ar[d]_f & Z' \ar[d]^{f'} \\
            \bP^1 \ar[r]^\simeq & \bP^1.
            }
        \end{equation*}
        Equivalently, there is an isomorphism $\phi\colon Z\simeq Z'$ satisfying $\phi^*F' = F$, where $F$ and $F'$ are the fibre classes of $f$ and $f'$, respectively.
    \end{enumerate}    
\end{definition}

Recall that a Fano threefold in Family 3-1 is a double cover of $Y = \bP^1\times \bP^1 \times \bP^1$ branched in an anticanonical K3 surface $Z\subset Y$. For such a K3 surface, the three projections $Z\to \bP^1$ are elliptic fibrations of multisection index $2$ with fibre classes given by $h_1,h_2$ and $h_3$, respectively. Our first goal is to show that $Z$ admits at most three elliptic fibrations up to isomorphism, provided $\NS(Z) \simeq \Lts,$ see Lemma \ref{lem: three elliptic fibrations} below.

\begin{lemma}\label{lem: three isotropic classes up to isometries}
    Let $F\in \Lto$ be a primitive isotropic class.
    Then there exist primitive isotropic classes $F', F''$ such that $F,F',F''$ is a basis of $\NS(Z)$ whose Gram matrix is 
    \[
    \begin{pmatrix}
        0 & 2 & 2 \\ 2 & 0 & 2 \\ 2 & 2 & 0
    \end{pmatrix}
    \]
    In particular, for each $1\leq i \leq 3$, there exists an isometry $\sigma\in O(\Lto)$ such that $\sigma(F) = h_i$.
\end{lemma}
\begin{proof}
    Firstly we note that $F$ has divisibility $2$ (indeed, every primitive class in $\Lto$ has divisibility $2$). Therefore, we may find a primitive class $D'$ such that $F\cdot D' = 2$. Note that $D'^2 = 4n$ for some $n\in \bZ$, thus $(D' - nF)^2 = 0$. Thus, denoting $F'\coloneqq D'-nF$, we now have a primitive sublattice $\langle F, F'\rangle \subset \Lto$ with Gram matrix 
    \[
    \begin{pmatrix}
        0 & 2 \\ 2 & 0
    \end{pmatrix}.
    \]
    We may now choose a second primitive class $D''\in \Lto$ such that $F,F',D''$ is a $\bZ$-basis for $\Lto$. Its Gram matrix is given by 
    \[ M\coloneqq 
    \begin{pmatrix}
        0 & 2 & 2a \\
        2 & 0 & 2b \\
        2a & 2b & 4c
    \end{pmatrix}
    \]
    for some $a,b,c\in \bZ$ such that $\det M = 16ab - 16c = 16$, hence $c = ab-1$. Now note that 
    \[
    F'' \coloneqq (1-a)F + (1-b)F' + D''
    \]
    satisfies $F\cdot F'' = F'\cdot F'' = 2$ and $F''^2 = 0$, and $F,F',F''$ is a $\bZ$-basis of $\Lto$. Therefore, the group automorphism $\sigma \in \Aut(\Lto)$ defined by $\sigma(F) = h_1$, $\sigma(F') = h_2$ and $\sigma(F'') = h_3$ is an isometry.
    This finishes the proof, since there is a natural action of the symmetric group $S_3$ on $\Lto$ by interchanging the basis vectors of $\Lto$.
\end{proof}

Next,we prove a technical lemma about isotropic subgroups of the discriminant group $\Ato \coloneqq A_{\Lto}$. Note that the order of the discriminant group is $|\Ato| = \det(\Lto) = 16$. This lemma will also be relevant in Section \ref{sec: FM partners of the branch}, where we study Fourier--Mukai partners of the branch divisors.
\begin{lemma} \label{lem: isotropic subgroups of Ato}
    The discriminant group $\Ato$ contains precisely $3$ isotropic subgroups of order $2$, and does not contain any isotropic subgroups of order $4$. The three isotropic subgroups of order $2$ are those generated by $\frac{1}{2}h_1,$ $\frac{1}{2}h_2,$ and $\frac{1}{2}h_3$, respectively.
\end{lemma}
\begin{proof}
    The discriminant group $\Ato$ is generated by the elements
    \[
    \begin{array}{ll}
         h_1^* &= \frac{1}{4}(-h_1 + h_2 + h_3)  \\
         h_2^* &= \frac{1}{4}(h_1 - h_2 + h_3)  \\
         h_3^* &= \frac{1}{4}(h_1+ h_2 - h_3).
    \end{array}
    \]
    However, these three elements are not independent in $\Ato$, and a more convenient generating set of $\Ato$ is
    \[
    \frac{1}{2}h_1, \qquad \frac{1}{2}h_2, \qquad \frac{1}{4}(h_1 + h_2 - h_3).
    \]
    Indeed, we have 
    \begin{equation} \label{eq: Ato generators}
    \Ato = \left\langle \frac{1}{2}h_1 \right\rangle \oplus \left\langle \frac{1}{2}h_2 \right\rangle \oplus \left\langle \frac{1}{4}(h_1 + h_2 - h_3) \right\rangle \simeq \bZ/2\bZ \oplus \bZ/2\bZ \oplus \bZ/4\bZ.
    \end{equation}
    Since $h_1$, $h_2$ and $h_3$ are isotropic, the subgroups 
    \begin{equation} \label{eq: Ato isotropic subgroups}
        \left\langle \frac{1}{2}h_1 \right\rangle, \quad \left\langle \frac{1}{2}h_2 \right\rangle, \quad \left\langle \frac{1}{2}h_3 \right\rangle
    \end{equation} are isotropic subgroups of order $2$ in $\Ato$. Moreover, using \eqref{eq: Ato generators}, we see that $\Ato$ has precisely $6$ subgroups of order $2$, and it is not difficult to show that precisely three of them are isotropic, namely those appearing in \eqref{eq: Ato isotropic subgroups}. Finally, a straightforward computation shows that $\Ato$ has no isotropic subgroups of order $4$. 
\end{proof}

We are now ready to prove that a K3 surface $Z$ with $\NS(Z)\simeq \Lto$ has at most three elliptic fibrations up to isomorphism.

\begin{lemma}\label{lem: three elliptic fibrations}
    Let $Z$ be a K3 surface with $\NS(Z) \simeq \Lto$. Suppose $Z\to \bP^1$ is an elliptic fibration with fibre class $F \in \NS(Z)$. 
    Then there is an isomorphism $\phi\colon Z\simeq Z$ such that $\phi^*F = h_i$ for some $1\leq i \leq 3$. 
\end{lemma}
\begin{proof}
    By Lemma \ref{lem: three isotropic classes up to isometries}, we may find isotropic classes $F',F''\in \NS(Z)$ such that $F,F',F''$ generate $\NS(Z)$, and such that their Gram matrix is precisely that of $\Lto$. It follows that $\frac{1}{2}F, \frac{1}{2}F', \frac{1}{2}F''$ generate three different isotropic subgroups of order $2$ in $\Ato$. By Lemma \ref{lem: isotropic subgroups of Ato}, these three subgroups are the subgroups generated by $\frac{1}{2}h_i$ for $1\leq i \leq 3$. We will assume
     \begin{equation*}
        \left\langle \frac{1}{2}F \right\rangle=\left\langle \frac{1}{2}h_1 \right\rangle, \quad \left\langle \frac{1}{2}F' \right\rangle= \left\langle \frac{1}{2}h_2 \right\rangle, \quad \left\langle \frac{1}{2}F'' \right\rangle= \left\langle \frac{1}{2}h_3 \right\rangle
    \end{equation*}
    as the other 5 cases follow by symmetry. Now let $\sigma\in O(\NS(Z))$ be the isometry with $\sigma(F) = h_1$, $\sigma(F')=h_2$ and $\sigma(F'') = h_3$. Then $\sigma$ preserves the ample cone, and using \eqref{eq: Ato generators} one sees that $\overline{\sigma} = \pm \id_{A_{\NS(Z)}}$, hence $\sigma = \phi^*$ for some automorphism $\phi\in \Aut(Z)$, which finishes the proof.
\end{proof}

\begin{remark}
    Similarly to Remark \ref{rmk: at most two, but usually two}, a K3 surface $Z$ with $\NS(Z)\simeq \Lto$ has precisely three elliptic fibrations up to isomorphism if $O_{\mathrm{Hodge}}(T(Z)) \simeq \bZ / 2 \bZ$, but this number may drop for special K3 surfaces which admit more Hodge isometries of $T(Z)$.
\end{remark}

As an immediate consequence of Lemma \ref{lem: three elliptic fibrations}, we obtain the following result.
\begin{lemma}\label{lem: isomorphism action on NS 3-1}
    Let $Z,Z' \in |-K_Y|$ be K3 surfaces with $\NS(Z) \simeq \Lto \simeq \NS(Z')$. Suppose we have $Z\simeq Z'$. Then there exists an isomorphism $\phi\colon Z\simeq Z'$ and a permutation $\sigma\in S_3$ such that $\phi^*(h_i') = h_{\sigma(i)}$ for all $1\leq i \leq 3$.
\end{lemma}

Let $Z$ be the branch divisor of a $Z$-general Fano threefold $X$ in the family $\mathcal{X}_{3\text{-}1}$. Then the three effective, primitive, isotropic classes $h_i\in \NS(Z)$ induce three elliptic fibrations $f_i\colon Z\to \bP^1$ of multisection index $2$ (see Definition \ref{def: elliptic fibrations}). We note that the inclusion $Z\hookrightarrow \bP^1 \times \bP^1 \times \bP^1$ is simply the product of these three elliptic fibrations. 

\begin{proof}[Proof of Proposition \ref{prop: branch determines cover 3-1}]
    Let $\phi\colon Z\simeq Z'$ be an isomorphism. By Lemma \ref{lem: isomorphism action on NS 3-1}, it follows that there exists a permutation $\sigma\in S_3$ such that $\phi^*(h_i') = h_{\sigma(i)}$ for all $1\leq i \leq 3$. In particular, for each $1\leq i \leq 3$, there is a commutative diagram
    \[
    \xymatrix{
        Z \ar[d]_{f_{\sigma(i)}}\ar[r]^\simeq_\phi & Z' \ar[d]^{f_i'} \\
        \bP^1 \ar[r]^\simeq_{\psi_i} & \bP^1.
    }
    \]
    Denote 
    \[
    \function{\phi_\sigma}{\bP^1 \times \bP^1 \times \bP^1}{\bP^1 \times \bP^1 \times \bP^1}{(x_1,x_2,x_3)}{(x_{\sigma(1)}, x_{\sigma(2)}, x_{\sigma(3)})}
    \]
    and $$\psi \coloneqq (\psi_1\times \psi_2 \times \psi_3) \circ \phi_\sigma.$$
    Then we see that there is a commutative diagram
    \[
    \xymatrix{
    Z \ar[r]^\simeq_\phi \ar@{^(->}[d]_{i_Z} & Z'\ar@{^(->}[d]^{i_{Z'}} \\
    \bP^1 \times \bP^1 \times \bP^1 \ar[r]^\simeq_{\psi} & \bP^1 \times \bP^1 \times \bP^1,
    }
    \]
    which shows that we have an isomorphism $X\simeq X'$.
\end{proof}

\subsection{Verra fourfolds} In this subsection, we prove the following proposition, which was also noted in \cite{Ver}.

\begin{proposition} \label{prop: branch determines cover V4}
    Let $V_4$ and $V_4'$ be Verra fourfolds branched in ordinary Verra threefolds $V_3$ and $V_3'$, respectively. Then, if $V_3$ and $V_3'$ are isomorphic, there is an isomorphism $V_4 \simeq V_4'$.
\end{proposition}

Recall that an ordinary Verra threefold is a bidegree $(2,2)$-divisor $V_3$ in $\bP^2\times \bP^2$. For an ordinary Verra threefold, the two projection maps $p_i\colon V_3 \to \bP^2$ are conic bundles. We now prove Proposition \ref{prop: branch determines cover V4} by showing that these are the only conic bundle structures on $V_3$, up to isomorphism.

\begin{proof}[Proof of Proposition \ref{prop: branch determines cover V4}]
    Suppose that we have an isomorphism $f\colon V_3 \simeq V_3'$. Then $f$ induces an isomorphism $f^*\colon \NS(V_3') \simeq \NS(V_3)$. Since $V_3$ and $V_3'$ are standard by the Lefschetz hyperplane theorem, we have $\NS(V_3) \simeq \bZ H_1 \oplus \bZ H_2$, where $H_i \coloneqq \pi_i^*L$ for a line $L\in \NS(\bP^2)$. 
    By Proposition \ref{prop: branch description V4}, the intersection form on $\NS(V_3)$ is given by:
    \[
    (aH_1+bH_2)^3 = 6ab(a+b).
    \]

   Therefore, the only primitive isotropic classes in $\NS(V_3)$, up to a sign, are $H_1$, $H_2$, and $H_1 - H_2$. However, it is easy to see that $\pm(H_1-H_2)$ is not nef, so $f^*H_1'$ is either $H_1$ or $H_2$ and similarly $f^*H_2'$ is either $H_1$ or $H_2$. 

    Suppose that $f^*H_1' = H_1$ and $f^*H_2' = H_2$. It suffices to prove the claim in this case, since we have an automorphism of $\bP^2 \times \bP^2$ which swaps the two factors.
    
    For each $i = 1,2$ we have a commutative square
    \[
    \xymatrix{
    V_3 \ar[r]_f^\simeq \ar[d]_{p_i} & V_3' \ar[d]^{p_i'} \\
    \bP^2 \ar[r]^\simeq_{\varphi_i} & \bP^2.
    }
    \]
    Putting these two squares together, we obtain a commutative square where the horizontal maps are isomorphisms and the vertical maps are the inclusion maps:
   \begin{equation} \label{eq: commutative square with verra threefolds}
        \xymatrix{
            V_3 \ar[r]_f^\simeq \ar[d]_{(p_1\times p_2)} & V_3' \ar[d]^{(p_1'\times p_2')} \\
            \bP^2 \times \bP^2 \ar[r]^\simeq _{(\varphi_1\times \varphi_2)} & \bP^2\times \bP^2.
        }
    \end{equation}
    This shows that $V_4 \simeq V_4'$. 
\end{proof}

\section{Equivalences of the branch loci} \label{sec: FM partners of the branch}
The main goal of this section is to prove the following result.

\begin{theorem} \label{thm: the k3s have no FM partners}
    Let $X$ be a $Z$-general Fano variety in one of the families $\mathcal{X}_{2\text{-}6(b)}, \mathcal{X}_{2\text{-}8}$, and $\mathcal{X}_{3\text{-}1}$. Let $Z$ be the associated K3 surface of $X$. Then $Z$ has no non-trivial Fourier--Mukai partners. That is, for any K3 surface $Z'$, we have 
    \[
    \Db(Z) \simeq \Db(Z') \implies Z\simeq Z'.
    \]
\end{theorem}
\begin{proof}
    Combine Proposition \ref{prop: the branch has no partners 2-6 2-8} and Proposition \ref{prop: the branch has no partners 3-1} below.
\end{proof}

Recall that $X$ is said to be \textit{$Z$-general} if the N\'eron--Severi lattice of $Z$ is isometric to $\Lts, \Lte$, or $\Lto$, see Definition \ref{def: Z general}.

\subsection{Families 2-6(b) and 2-8}
In this subsection, we prove Theorem \ref{thm: the k3s have no FM partners} for the families 2-6(b) and 2-8. 
We use the Counting Formula for Fourier--Mukai partners of K3 surfaces of \cite{HLOY}, see Theorem \ref{thm: counting formula}.

We first show that $\Lts$ and $\Lte$ are each unique in their genus. 
This result was also verified using the OSCAR function 
\texttt{genus\_representatives}
developed by Simon Brandhorst and Stevell Muller \cite{OSCAR, OSCAR-book}.

To compute the genera of $\Lts$ and $\Lte$, we use the well-known correspondence between lattices and quadratic forms. Our main reference for binary quadratic forms is \cite[Chapter 15]{CS98}. 
Explicitly, the indefinite even lattice 
\[
L = \begin{pmatrix}
    2a & b \\ b & 2c
\end{pmatrix}
\]
corresponds to the binary quadratic form $f_L\coloneqq ax^2 + bxy + cy^2$. The discriminant of $f_L$ is defined to be $\disc(f_L)\coloneqq b^2 - 4ac$ (note that this coincides with our definition of $\disc(L)$ in this case). The form $f_L$ is said to be a \textit{reduced binary quadratic form} if the following inequalities are satisfied:
\begin{equation}\label{eq: reduced inequalities for quadratic forms}
    0 < b < \sqrt{\disc(f_L)} < \mathrm{min}\left(b + |a|, b + |c|\right).
\end{equation}
It is well-known that every indefinite binary quadratic form is properly equivalent to at least one reduced binary quadratic form, see \cite[\S15.3.3]{CS98}.
Using this fact, we will prove the following result for $\Lte$:

\begin{lemma} \label{lem: Lte is the only one of its discriminant}
    The lattice $\Lte$ is the unique indefinite even lattice of rank $2$ and discriminant $8$. In particular, $\Lte$ is unique in its genus.
\end{lemma}
\begin{proof}
    We use the correspondence between binary quadratic forms and even lattices of rank 2. For the first statement, it suffices to show that there exists a unique reduced binary quadratic form of discriminant $8$, up to equivalence.
    Let $f = ax^2 + bxy + cy^2$ be a reduced binary quadratic form of discriminant $\disc(f) = 8$. By \eqref{eq: reduced inequalities for quadratic forms}, we have $b = 1$ or $b = 2$. However, since $b^2-4ac = 8$, we have $b^2 \equiv 0 \pmod 4$, thus it follows that $b = 2$. Again using $b^2 -4ac = 8$, we find $ac = -1$, and so $(a,c) = (\pm 1, \mp 1)$. However, the binary quadratic forms $x^2 + 2xy - y^2$ and $-x^2 + 2xy + y^2$ are clearly properly equivalent, so there exists a unique equivalence class of reduced binary quadratic forms of discriminant $8$, as required.

    For the final assertion, suppose $L$ is any lattice in the same genus as $\Lte$. Then $L$ is an even, indefinite lattice of rank $2$ and discriminant $\disc(L) = 8$, so we have $L\simeq \Lte$ by the above discussion.
\end{proof}

For $\Lts$, the situation is slightly more delicate:
\begin{lemma}\label{lem: Lts is the only one of its discriminant}
    There are precisely $2$ even, indefinite lattices of rank $2$ and discriminant $12$, and they are not in the same genus. In particular, $\Lts$ is unique in its genus.
\end{lemma}
\begin{proof}
    We use the same strategy as in the proof of Lemma \ref{lem: Lte is the only one of its discriminant}. If $f = ax^2 + bxy + cy^2$ is a reduced binary quadratic form of discriminant $12$, then \eqref{eq: reduced inequalities for quadratic forms} implies that $b = 1$, $b=2$, or $b=3$. Once again we can reduce the equality $b^2 - 4ac = 12$ modulo $4$ to rule out $b=1$ and $b=3$. Thus, we have $b = 2$ and $ac = -2$. This leads to two possibly inequivalent reduced binary quadratic forms, whose lattices are given by 
    \[
    L_1 \coloneqq \begin{pmatrix}
        2 & 2 \\ 2 & -4
    \end{pmatrix},
    \qquad L_2 \coloneqq \begin{pmatrix}
        -2 & 2 \\ 2 & 4
    \end{pmatrix}.
    \]
    We claim that $L_1 \not \simeq L_2$. In fact, we will show that $L_1$ and $L_2$ are not in the same genus, because they do not have isometric discriminant groups. 
    Firstly, note that $A_{L_1}$ and $A_{L_2}$ are isomorphic as groups, as they are each isomorphic to the group $\bZ/2\bZ \oplus \bZ/6\bZ$. A straightforward but technical check shows that the $3$-primary part $A_{L_1}^{(3)}$ of $A_{L_1}$ is a group of order $3$ for which any generator $x\in A_{L_1}^{(3)}$ has $q(x) = \frac{4}{3}\pmod{2\bZ}$, whereas any generator $y\in A_{L_2}^{(3)}$ has $q(y) = \frac{2}{3} \pmod{2\bZ}.$ Thus, we find $A_{L_1}\not \simeq A_{L_2}$, hence $L_1$ and $L_2$ are not in the same genus, and there exist precisely two even, indefinite lattices of rank $2$ and discriminant $12$. 
\end{proof}

\begin{remark}
    One can check that $L_1$ is isometric to $\Lts$. Indeed, if $v,w\in \Lts$ are basis vectors whose Gram matrix is $$\begin{pmatrix}
        2 & 4 \\ 4 & 2
    \end{pmatrix},$$
    then $v, w-v\in \Lts$ are basis vectors whose Gram matrix is 
    $$\begin{pmatrix}
        2 & 2 \\ 2 & -4
    \end{pmatrix}.$$
\end{remark}

We immediately obtain the following.
\begin{proposition} \label{prop:unique_in_genus}
    The lattices $\Lts$ and $\Lte$ are each unique in their respective genera. 
\end{proposition}
\begin{proof}
    Combine Lemma \ref{lem: Lte is the only one of its discriminant} and Lemma \ref{lem: Lts is the only one of its discriminant}.
\end{proof}

\begin{proposition} \label{prop: the branch has no partners 2-6 2-8}
    Let $Z$ be a K3 surface for which $\NS(Z)$ is isometric to either $\Lts$ or $\Lte$. Then $Z$ has no non-trivial Fourier--Mukai partners.
\end{proposition}
\begin{proof}
    Since $\Lts$ and $\Lte$ are unique in their genera, the Counting Formula of Theorem \ref{thm: counting formula} consists of a single term:
    \begin{equation} \label{eq: double quotient for FM number}
    |\FM(Z)| = \left|O(\NS(Z))\backslash O(A_{\NS(Z)}) / O_{\mathrm{Hodge}}(T(Z))\right|.
    \end{equation}
    With a straightforward computation, one can show that $O(A_{\NS(Z)}) = \left\{\pm \id\right\}$, thus the natural map $O(\NS(Z)) \to O(A_{\NS(Z)})$ is surjective and the double quotient on the right-hand side of \eqref{eq: double quotient for FM number} is trivial. Therefore, we have $|\FM(Z)| = 1$, as required. 
\end{proof}

\subsection{Family 3-1}
Our strategy for proving Theorem \ref{thm: the k3s have no FM partners} for Family 3-1 is inspired by \cite{MS24}. We note that the branch divisor $Z$ is an elliptic K3 surface of Picard rank 3. In that paper, only elliptic K3 surfaces of Picard rank 2 are considered, but we will show that the techniques can be extended to $Z$.

We begin by recalling certain facts from \cite{MS24}. To any elliptic K3 surface $f\colon S\to \bP^1$, we can associate a relative Jacobian which we denote by $\Jac^0(S)\to \bP^1$. 
Then, by \cite[Lemma 2.8]{MS24}, we have a short exact sequence
\begin{equation} \label{eq: transcendental brauer sequence}
     0 \longrightarrow T(S) \longrightarrow T(\Jac^0(S)) \overset{\alpha_{S}}{\longrightarrow} \bZ/t\bZ \longrightarrow 0,
\end{equation}
where $t$ is the multisection index of $f$. Via Lemma \ref{lem: overlattice subgroup correspondence}, the quotient $ T(\Jac^0(S))/T(S)$ can be seen as a subgroup of $A_{\NS(S)}$. Explicitly, this subgroup is cyclic and isotropic, generated by the class $\frac{1}{t}F \in A_{\NS(S)}$, where $F$ is the fibre class of the elliptic fibration $f$. 

Moreover, the map $\alpha_S$ of \eqref{eq: transcendental brauer sequence} is the element of $\Br(\Jac^0(S)) \simeq \Hom(T(\Jac^0(S)),\bQ/\bZ)$ which corresponds to the class of $S$ via the natural isomorphism $\Br(S) \simeq \Sha(\Jac^0(S))$. Here $\Sha(\Jac^0(S))$ denotes the \textit{Tate--\v{S}afarevi\v c group} of $\Jac^0(S)$.

\begin{proposition} \label{prop: the branch has no partners 3-1}
    Let $Z$ be a K3 surface with $\NS(Z)\simeq \Lto$. Then $Z$ has no non-trivial Fourier--Mukai partners.
\end{proposition}
\begin{proof}
    Suppose $Z'$ is a Fourier--Mukai partner of $Z$. We wish to show that $Z'\simeq Z$. 

    Firstly, note that $Z'$ admits an elliptic fibration $f\colon Z'\to \bP^1$ by \cite[Proposition 16]{HT17}. Denote by $\Jac^0(Z')$ the relative Jacobian of $f$. 
    Fix any Hodge isometry $\phi\colon T(Z)\simeq T(Z')$, which exists due to the derived Torelli theorem, see Theorem \ref{thm: derived torelli theorem}. Then the composition 
    \[
    T(Z) \overset{\phi}{\simeq} T(Z') \hookrightarrow T(\Jac^0(Z'))
    \]
    is an overlattice of $T(Z)$, hence $H\coloneqq T(\Jac^0(Z'))/T(Z)$ naturally defines an isotropic subgroup of $A_{\NS(Z)}$ via Lemma \ref{lem: overlattice subgroup correspondence}. 
    Moreover, we have $$16 = \det(T(Z)) = |H|^2 \det(T(\Jac^0(Z))),$$ thus $H$ has order $1, 2$ or $4$. However, $A_{\NS(Z')}\simeq \Ato$ has no isotropic subgroups of order $4$ by Lemma \ref{lem: isotropic subgroups of Ato}. 
    On the other hand if $H$ is the trivial group, then the multisection index of $f$ is $1$. In this case $Z'$ has no non-trivial Fourier--Mukai partners by \cite[Corollary 2.7]{HLOY}. However, by assumption, $Z$ is a Fourier--Mukai partner of $Z'$ and $Z$ admits no elliptic fibrations of multisection index $1$, a contradiction. We conclude that $H$ must be an isotropic subgroup of order $2$. 
    By Lemma \ref{lem: isotropic subgroups of Ato}, $A_{\NS(Z)}$ has precisely $3$ isotropic subgroups of order $2$, namely those appearing in \eqref{eq: Ato isotropic subgroups}. We assume that $H = \langle \frac{1}{2}h_1\rangle$, as the other two cases follow by symmetry. 
    Then, if we denote by $\Jac^0(Z)$ the relative Jacobian of the elliptic fibration $f_1\colon Z \to \bP^1$ defined by $h_1$, we have a commutative diagram  by Lemma \ref{lem: overlattice subgroup correspondence}:
    \begin{equation} \label{eq: diagram relating brauer classes}
    \xymatrix{
    0 \ar[r]& T(Z) \ar[r] \ar[d]_\simeq^{\phi} & T(\Jac^0(Z)) \ar[r]^{\alpha_Z} \ar[d]_\simeq^{\tilde{\phi}}  & \bZ/2\bZ \ar[r] \ar[d]_\simeq & 0 \\
    0 \ar[r] & T(Z') \ar[r]  & T(\Jac^0(Z')) \ar[r]^{\alpha_{Z'}}  & \bZ/2\bZ \ar[r]  & 0.
    }
    \end{equation}
    Since $Z$ has Picard rank $3$, we have $O_{\mathrm{Hodge}}(T(Z)) = \left\{\pm \id\right\}$ by Lemma \ref{lem: odd picard rank few hodge isometries}. Hence, up to a sign, $\tilde{\phi}$ is the pullback along a group isomorphism $g\colon \Jac^0(Z')\simeq \Jac^0(Z)$. We see from \eqref{eq: diagram relating brauer classes} that the isomorphism $g$ satisfies $g^*\alpha_{Z} = \alpha_{Z'},$ thus we have $Z\simeq Z'$ by \cite[Proposition 4.8]{MS24}.
\end{proof}

\subsection{Verra fourfolds} \label{sec: Verra fourfolds}

Recall that for a Verra fourfold $V_4$, the two quadric surface fibrations $V_4 \to \bP^2$ defined by the projection maps each determine a twisted K3 surface $(S,\alpha)$, where $\alpha\in \Br(S)_2$ is some $2$-torsion Brauer class, see \cite[Lemma 4.2]{Kuzcub}.
This twisted K3 surface is obtained as follows.
By Theorem \ref{thm: Kuz quadric fibrations SOD}, the derived category of $V_4$ has a semiorthogonal decomposition of the form
\[
\Db(V_4) = \langle \Db(\bP^2,\cB_0), \cO(-1,0),\cO,\cO(1,0),\cO,(0,1),\cO(2,1)\rangle.
\]
We call $\cA_{V_4} \coloneqq \Db(\bP^2,\cB_0)$ the Kuznetsov component of $V_4$. There is an equivalence $\Db(S,\alpha) \simeq \Db(\bP^2,\cB_0)$ by \cite[Lemma 4.2]{Kuzcub}. 
In this subsection, we study the twisted Fourier--Mukai partners of the twisted K3 surface $(S,\alpha)$.

We first briefly recall the theory developed in \cite[\S 9.8]{vGe} about $2$-torsion Brauer classes on K3 surfaces.
If $S$ is a K3 surface with $\NS(S) = \langle h \rangle  = \langle 2 \rangle$, then $\phi_h \colon S \to \bP^2$ is a sextic structure on $S$. Letting $C\subset \bP^2$ be the branch sextic, we have a short exact sequence 
\[
0 \to \Br(\bP^2 \setminus C)_2 \overset{\phi_h^*}{\to} \Br(S)_2 \to \bZ/2\bZ \to 0.
\]

\begin{definition}
    A Brauer class $\alpha \in \Br(S)_2$ of order $2$ on a K3 surface $S$ of degree $2$ and Picard rank $\rho = 1$ is said to be of \textit{Verra type} if $\alpha$ is contained in the image of $\phi_h^*$.
\end{definition}

The reason we say that these Brauer classes are of Verra type is the following.
Let $\alpha\in \Br(S)$ be a Brauer class of Verra type, and let $\alpha'\in \Br(\bP^2 \setminus C)$ be the Brauer class which satisfies $\alpha = \phi_h^*\alpha'$. Then $\alpha'$ defines a conic bundle $V_\alpha \to \bP^2$ whose discriminant locus is $C$, and this conic bundle is an ordinary Verra threefold. We call $V_\alpha$ the ordinary Verra threefold \textit{associated with} the twisted K3 surface $(S,\alpha)$.

If $V_4$ is a Verra fourfold, then the two twisted K3 surfaces $(S_1,\alpha_1)$, $(S_2,\alpha_2)$ turn out to be of Verra type. Moreover, we now show that the Verra threefolds associated with them are isomorphic, and in fact isomorphic to the branch locus $V_3$ of the double cover $V_4 \to \bP^2\times \bP^2$:

\begin{lemma}\label{lem: brauer classes match}
    Let  $V_4$ be a Verra fourfold with branch locus $V_3$. Then the quadric surface fibrations $V_4 \to \bP^2$ defined by the projection maps induce two twisted K3 surfaces of Verra type $(S_1,\alpha_1)$, $(S_2,\alpha_2)$, and we have $V_{\alpha_1} \simeq V_3 \simeq V_{\alpha_2}$.
\end{lemma}
\begin{proof}
    We consider only the quadric surface fibration $\pi_1\colon V_4\to \bP^2$, as the situation is completely symmetric. Let $C\subset \bP^2$ be the discriminant locus of $\pi_1$. The conic bundle $\pi_1|_{V_3}\colon V_3 \to \bP^2$ induces an element $\alpha'\in \Br(\bP^2\setminus C)$, which we pull back to $S_1$ via $\phi_h^*$. We must show that we have 
    \[
    \alpha_1 = \phi_h^*\alpha' \in \Br(S).
    \]
    We note that it suffices to work over generic fibres, since there is a natural inclusion $\Br(S) \hookrightarrow \Br(\bC(S))$.
    Let $K = \bC(\bP^2)$, and let $L = \bC(S)$.
    The conic bundle $V_3\to \bP^2$ is given generically by a $3$-dimensional quadratic form $q_3$ over $K$. Its even Clifford algebra $\cB_0(q_3)$ is a quaternion algebra, whose Brauer class is $\alpha'\in \Br(K)$ by construction. 
    For the Verra fourfold, the generic fibre of $V_4\to \bP^2$ is given by the quadratic form $q_4 = \langle 1\rangle \oplus q_3$, and it follows that we have
    $\cB_0(q_4) \simeq \cB(q_3)$, where $\cB(q_3)$ denotes the full Clifford algebra of $q_3$. Since $q_3$ has odd dimension, the centre of $\cB(q_3)$ is $L$, and moreover $\cB(q_3) \simeq \cB_0(q_3)\otimes_K L$ as $L$-algebras, so we have the following equalities in $\Br(L)$:
    \[
    \alpha_1 = [\cB_0(q_4)] = [\cB(q_3)] = [\cB_0(q_3) \otimes_K L] = \phi_h^*[\cB_0(q_3)] = \phi_h^*\alpha',
    \]
    which is what we set out to prove.
\end{proof}

We now count the twisted Fourier--Mukai partners of $(S,\alpha)$ using the counting formula of \cite{Ma}. For a twisted K3 surface $(S,\alpha)$, we write
\[
\FM(S,\alpha) \coloneqq \left\{(S',\alpha') \mid \Db(S,\alpha) \simeq \Db(S',\alpha')\right\}/_\simeq
\]
where two twisted K3 surfaces $(S',\alpha')$ and $(S'',\alpha'')$ are said to be isomorphic if there is an isomorphism $f\colon S'\simeq S''$ such that $f^*\alpha'' = \alpha'$.
\begin{lemma}\label{lem: the twisted k3s have few partners}
    Let $S$ be a K3 surface of degree $2$ and Picard rank $\rho = 1$. Let $\alpha \in \Br(S)_2$ be a Brauer class of Verra type. Then we have $|\FM(S,\alpha)| \leq 2$.
\end{lemma}
\begin{proof}
    Recall from \cite[\S9.8]{vGe} that the discriminant form of the twisted transcendental lattice $T(S,\alpha)$ is given by 
    \[
    \functionstar{(\bZ/2\bZ)^3}{\bQ/2\bZ}{(\overline{x},\overline{y},\overline{z})}{\frac{1}{2}x^2 +  yz \pmod{2\bZ}.}
    \]
    From this, combined with \cite[Lemma 3.2]{Ma}, it follows that any Fourier--Mukai partner $(S',\alpha')$ of $(S,\alpha)$ satisfies $\alpha'\in \Br(S')_2$. Moreover, $(S,\alpha)$ has no untwisted Fourier--Mukai partner, since $T(S,\alpha)$ cannot be primitively embedded in the K3-lattice.
    Since $A_{T(S,\alpha)}$ contains precisely two non-zero isotropic vectors, namely $(0,1,0)$ and $(0,0,1)$, we may use \cite[Lemma 3.2]{Ma} to see that $|\FM(S,\alpha)|\leq 2$. 
\end{proof}

Let $\alpha\in \Br(S)$ be a Brauer class of Verra type. Since ordinary Verra threefolds have two conic bundle structures defined by the projections to $\bP^2$, the Verra threefold $V_\alpha$ has two twisted K3 surfaces it is associated with, namely $(S,\alpha)$ and some other twisted K3 surface $(S',\alpha')$. Moreover, we have
\[
\Db(S,\alpha) \simeq \cA_{V_4} \simeq \Db(S',\alpha'),
\]
where $V_4$ is the Verra fourfold that has branch locus $V_\alpha$ \cite[Theorem 1.1]{ADPZ}.
If $V_\alpha$ is a very general ordinary Verra threefold, then $(S',\alpha')$ is the unique non-trivial Fourier--Mukai partner of $(S,\alpha)$, since we have $S\not \simeq S'$ by \cite{KKM}.

\begin{proposition} \label{prop: the derived category of the verra twisted K3 determines the Verra threefold}
    Let $S$ and $S'$ be two K3 surfaces of degree $2$ with $\rho = 1$. Let $\alpha \in \Br(S)$ and $\alpha'\in \Br(S')$ be Brauer classes of Verra type whose associated ordinary Verra threefolds $V_\alpha$ and $V_{\alpha'}$ are very general. Then, if there is an equivalence $\Db(S,\alpha)\simeq \Db(S',\alpha')$, it follows that we have $V_\alpha \simeq V_{\alpha'}$.
\end{proposition}
\begin{proof}
    By the above discussion, $V_\alpha$ and $V_{\alpha'}$ are each associated with 2 twisted K3 surfaces. The assumption $\Db(S,\alpha)\simeq \Db(S',\alpha')$ implies that $V_\alpha$ and $V_{\alpha'}$ have the same associated twisted K3 surfaces, hence they are isomorphic. 
\end{proof}

\section{Equivariant Kuznetsov components and semiorthogonal decompositions} \label{sec: Kuznetsov components and SODs}
We now study the equivariant Kuznetsov components of the Fano varieties in the three families of Fano threefolds under consideration. The results for Fano threefolds in this section hold in greater generality than Theorem \ref{thm: branch determines Fano} and Theorem \ref{thm: the k3s have no FM partners}, which hold for $Z$-general Fano threefolds. 
Let $X$ be a Fano threefold in one of the families $\mathcal{X}_{2\text{-}6(b)}, \mathcal{X}_{2\text{-}8}$, and $\mathcal{X}_{3\text{-}1}$. Recall from Definition \ref{def: Z general} that we say that $X$ is \textit{K3-general} if the branch locus $Z$ of the double cover $X\to Y$ is a smooth K3 surface.
Note that $Z$-general Fano threefolds are always K3-general, but the converse does not hold, see also Remark \ref{rem: Z general K3 general}.

The main goal of this section is to prove the following two results.
\begin{theorem}\label{thm: any equivalence lifts to equivariant categories}
    Let $X$ be a K3-general Fano variety in one of the families $\mathcal{X}_{2\text{-}6(b)}, \mathcal{X}_{2\text{-}8}$, and $\mathcal{X}_{3\text{-}1}$, and let $X'$ be a Fano variety in the same family. Then any equivalence $\cA_X\simeq \cA_{X'}$ lifts to an equivalence of equivariant categories $\cA_X^{\mu_2} \simeq \cA_{X'}^{\mu_2}$:
    \[
    \cA_X \simeq \cA_{X'} \implies \cA_{X}^{\mu_2} \simeq \cA_{X'}^{\mu_2}.
    \]
\end{theorem}
\begin{proof}
    This follows from Proposition \ref{prop:2-6b_serre_involution_relationship}, Proposition \ref{prop:2-8_serre_involution_relationship}, and Proposition \ref{prop:3-1 serre involution relationship} below, and \cite[Lemma 6.2]{DJR}. An analogous result was also used in the proof of \cite[Theorem 9.9]{JLLZ}.
\end{proof}

\begin{theorem}\label{thm: equivariant kuznetsov components are the branch loci}
    Let $X$ be a K3-general Fano variety in one of the families $\mathcal{X}_{2\text{-}6(b)}, \mathcal{X}_{2\text{-}8}$, and $\mathcal{X}_{3\text{-}1}$. Let $Z$ be the associated K3 surface of $X$. Then we have an equivalence
    \[
    \cA_X^{\mu_2}\simeq \Db(Z).
    \]
\end{theorem}
\begin{proof}
    This follows from Proposition \ref{prop:2-6b_equivariant_Ku}, Proposition \ref{prop:2-8_equivariant_Ku}, and Proposition \ref{prop:3-1_equivariant_Ku} below.
\end{proof}

\subsection{Family 2-6(b)}
Let $X$ be a K3-general Fano variety in Family 2-6(b). That is, $X$ is a double cover of a bidegree $(1,1)$-divisor $Y \subset \bP^2\times \bP^2$ branched in an anticanonical K3 surface $Z\subset Y$. Recall from Lemma \ref{lem: bidegree (1,1) divisors are projective bundles over the plane} that $Y$ can equivalently be described as $\bP(T_{\bP^2})$, the projectivised tangent bundle of the projective plane. By Orlov's Projective Bundle Formula \cite{OrlovBlowUp}, we have the semiorthogonal decomposition 
\begin{equation}\label{eq: projective bundle sod unrefined}
    \Db(\bP(T_{\bP^2})) = \langle \pi^* \Db(\bP^2), \pi^* \Db(\bP^2) \otimes \cO_{\bP(T_{\bP^2})}(1) \rangle 
\end{equation}
where $\pi \colon Y \to \bP^2$ is the projective bundle map, and $\cO_{\bP(T_{\bP^2})}(1)$ is the tautological bundle on the projectivisation. We note that \eqref{eq: projective bundle sod unrefined} is a rectangular Lefschetz decomposition with $\cB = \pi^*\Db(\bP^2)$, see Definition \ref{def: rectangular lefschetz decomposition}.

By Lemma \ref{lem: SOD from rectangular lefschetz base}, we have 
\begin{equation} \label{eq: 2-6(b) Kuznetsov component}
    \Db(X) = \langle \cA_X, \pi^* \Db(\bP^2) \rangle.
\end{equation}
Moreover, \cite[Theorem 1.1]{KP} directly applies, and gives us the following description of the equivariant Kuznetsov component.

\begin{proposition} \label{prop:2-6b_equivariant_Ku}
    We have 
    \begin{equation*}
        \cA_X^{\mu_2} \simeq \Db(Z) .
    \end{equation*}
\end{proposition}

We also have an explicit description of the involution autoequivalence $\tau_{\cA_X} \colon \cA_X \to \cA_X$ in this case:

\begin{proposition} \label{prop:2-6b_serre_involution_relationship}
    We have the isomorphism of functors 
    \begin{equation*}
        S_{\cA_X} \simeq \tau_{\cA_X}[2] .
    \end{equation*}
\end{proposition}

\begin{proof}
    This is by either \cite[Corollary 3.18]{KuzCY}, or \cite[Theorem 7.7 and Proposition 7.10]{KP}, or Proposition \ref{prop: IK prop 3.2}.
\end{proof}

\subsection{Family 2-8}
Let $X$ be a K3-general Fano variety in Family 2-8. That is, $X$ is the double cover of $Y \coloneqq \Bl_p(\bP^3)$ branched in an anticanonical K3 surface $Z\subset Y$. Recall that the image of the composition $Z\hookrightarrow Y \to \bP^3$ is a singular quartic K3 surface $Z'\subset \bP^3$ with an ordinary double point at $p$. If we let $X'$ be the double cover of $\bP^3$ branched in $Z'$, then we have $X = \Bl_p(X')$. In other words, we are in the setting of \cite[\S 3.2]{IK}.

If we denote by $H \in \NS(Y)$ the pullback of a hyperplane class, and by $E\in \NS(Y)$ the exceptional divisor, then the canonical bundle of $Y$ is $K_Y = 2E - 4H$. Let $D = 2H-E$, so that $K_Y = -2D$. Then, by \cite[Proposition 3.6]{IK}, we obtain the following semiorthogonal decomposition of $\Db(Y)$: 
\begin{align}
    \Db(Y) &= \langle \cO_Y(-3H+E), \cO_Y(-2H), \cO_Y(-2H+E), \cO_Y(-H), \cO_Y(-E), \cO_Y \rangle \label{eq: 2-8 Y SOD}\\
    &= \langle \cE_1(-D), \cE_2(-D),\cE_3(-D),\cE_1,\cE_2,\cE_3 \rangle, \notag
\end{align}
where 
\[
    \cE_1\coloneqq \cO_Y(-H),\quad \cE_2 \coloneqq \cO_Y(-E), \quad \cE_3 \coloneqq \cO_Y.
\]
Applying Proposition \ref{prop: double cover SODs IK}, we obtain the following semiorthogonal decomposition of $\Db(X)$. 
\begin{lemma} \label{lem:blow_up_sod}
    Let $f \colon X \to Y$ be as above. Then there is a semiorthogonal decomposition
    \begin{equation}
        \Db(X) = \langle \cA_X, \cO_X(-H), \cO_X(-E), \cO_X  \rangle.
    \end{equation}
\end{lemma}
Now, using Lemma \ref{thm:G action commutes through SOD}, we obtain the following semiorthogonal decomposition of the equivariant derived category.
\begin{lemma} \label{lem:second 2-8 equivariant SOD}
    We have the semiorthogonal decomposition
    \begin{align*}
        \Db(X)^{\mu_2} &= \langle \cA_X^{\mu_2},  \cO_X(-H) \rho_0, \cO_X(-E) \rho_0 , \cO_X \rho_0 ,  \cO_X(-H) \rho_1, \cO_X(-E) \rho_1, \cO_X \rho_1   \rangle . 
    \end{align*}
\end{lemma}

\begin{proof}
    Apply Lemma \ref{thm:G action commutes through SOD} to the semiorthogonal decomposition from Lemma \ref{lem:blow_up_sod}.
\end{proof}

Recall from Theorem \ref{thm: KP thm 4.1} that $\Db(X)^{\mu_2}$ admits another semiorthogonal decomposition given as $$\Db(X)^{\mu_2} = \langle f_0^* \Db(Y), j_{0 *} \Db(Z) \rangle.$$ Combining this with the semiorthogonal decomposition \eqref{eq: 2-8 Y SOD} of $\Db(Y)$, we obtain
\begin{align} 
        \Db(X)^{\mu_2} = \langle &f_0^* \Db(Y), j_{0 *} \Db(Z) \rangle \notag \\
        =\langle &\cO_X(-3H+E) \rho_0, \cO_X(-2H) \rho_0, \cO_X(-2H+E) \rho_0, \label{eq: 2-8 second equivariant SOD of X}\\
        &\cO_X(-H) \rho_0, \cO_X(-E) \rho_0, \cO_X \rho_0, j_{0*} \Db(Z) \rangle \notag.
\end{align}
We now show that the semiorthogonal decomposition of $\Db(X)^{\mu_2}$ of Lemma \ref{lem:second 2-8 equivariant SOD} is a mutation of \eqref{eq: 2-8 second equivariant SOD of X}, which allows us to compute the equivariant Kuznetsov component $\cA_X^{\mu_2}$. 

\begin{proposition}\label{prop:2-8_equivariant_Ku}
    We have 
    \begin{equation*}
        \cA_X^{\mu_2} \simeq \Db(Z) .
    \end{equation*}
\end{proposition}

\begin{proof}
    We apply Lemma \ref{lem:SODs and Serre functors} to the semiorthogonal decomposition \eqref{eq: 2-8 second equivariant SOD of X}, with
    \begin{equation*}
        \cA_1 = \langle \cO_X(-3H+E) \rho_0, \cO_X(-2H) \rho_0, \cO_X(-2H+E) \rho_0 \rangle 
    \end{equation*}
    and 
    \begin{equation*}
        \cA_2 = \langle \cO_X(-H) \rho_0, \cO_X(-E) \rho_0, \cO_X \rho_0, j_{0*} \Db(Z) \rangle .
    \end{equation*}
    Note that $S_{\Db(X)^{\mu_2}}(-) = - \otimes \cO_X(-2H+E) \rho_1$. Thus, we get
    \begin{align}
        \Db(X)^{\mu_2} = \langle &\cA_2, S_{\Db(X)^{\mu_2}}^{-1}(\cA_1)  \rangle \notag\\
        = \langle &\cO_X(-H) \rho_0, \cO_X(-E) \rho_0, \cO_X \rho_0 , j_{0*} \Db(Z) ,
        \cO_X(-H) \rho_1 , \cO_X(-E) \rho_1, \cO_X \rho_1 \rangle \notag\\
        = \langle &\bL_{\cO_X(-H) \rho_0, \cO_X(-E) \rho_0, \cO_X \rho_0} j_{0*} \Db(Z) , \cO_X(-H) \rho_0, \cO_X(-E) \rho_0, \cO_X \rho_0, \notag\\
        &\cO_X(-H) \rho_1 , \cO_X(-E) \rho_1, \cO_X \rho_1   \rangle . \notag
    \end{align}
    
    This exhibits $\bL_{\cO_X(-H) \rho_0, \cO_X(-E) \rho_0, \cO_X \rho_0} j_{0*} \Db(Z) \simeq \Db(Z)$ as the right-orthogonal of the subcategory
    \[
    \langle \cO_X(-H) \rho_0, \cO_X(-E) \rho_0, \cO_X \rho_0, \cO_X(-H) \rho_1 , \cO_X(-E) \rho_1, \cO_X \rho_1   \rangle,
    \]
    which is precisely $\cA_X^{\mu_2}$, by Lemma \ref{lem:second 2-8 equivariant SOD}.
\end{proof}

\begin{proposition} \label{prop:2-8_serre_involution_relationship}
    We have the isomorphism of functors 
    \begin{equation*}
        S_{\cA_X} \simeq \tau_{\cA_X}[2] .
    \end{equation*}
\end{proposition}

\begin{proof}
    This is by \cite[Corollary 3.7]{IK}.
\end{proof}

\subsection{Family 3-1}

Recall that the Fano threefolds $X$ in Family 3-1 are double covers of $Y = \bP^1 \times \bP^1 \times \bP^1$ branched in divisors $Z$ of tridegree $(2,2,2)$. In this subsection, we assume that $X$ is a K3-general Fano threefold in Family 3-1, so that $Z$ is a smooth K3 surface. We first compute some mutations that we will need later on.

\begin{lemma} \label{lem:3-1_mutations}
    We have 
    \begin{align*}
        &\bR_{\cO_Y(0,0,0)} \cO_Y(-1,0,0) \simeq \cO_Y(1,0,0)[-1] \\
        &\bR_{\cO_Y(0,0,0)} \cO_Y(0,-1,0) \simeq \cO_Y(0,1,0)[-1] \\
        &\bR_{\cO_Y(0,0,0)} \cO_Y(0,0,-1) \simeq \cO_Y(0,0,1)[-1] .
    \end{align*}
\end{lemma}

\begin{proof}
    We show the first isomorphism. For the mutation in question, the relevant Hom-space is the dual of
    \begin{align*}
        \Hom^\bullet(\cO_Y(-1,0,0), \cO_Y(0,0,0)) &= \Hom^\bullet(\cO_Y(0,0,0), \cO_Y(1,0,0)) \\
        &\simeq H^\bullet(\bP^1, \cO_{\bP^1}(1)) \otimes H^\bullet(\bP^1, \cO_{\bP^1}) \otimes H^\bullet(\bP^1, \cO_{\bP^1}) \\
        &= \bC^2[0]
    \end{align*}
    where for the second isomorphism, we have used the Künneth Theorem for sheaf cohomology. Hence, the right mutation fits into the triangle 
    \begin{equation*}
        \bR_{\cO_Y(0,0,0)} \cO_Y(-1,0,0) \to \cO_Y(-1,0,0) \to \cO_Y(0,0,0)^{\oplus 2} .
    \end{equation*}
    Now consider the Euler short exact sequence $0 \to \cO_{\bP^1}(-1) \to \cO_{\bP^1}^{\oplus 2} \to \cO_{\bP^1}(1) \to 0$ on the first $\bP^1$ factor. Pulling this back to $Y$ and comparing with the triangle above gives the desired isomorphism. The rest of the isomorphisms are similar. 
\end{proof}

\begin{lemma} \label{lem:3-1 D(Y) SOD}
    We have the semiorthogonal decomposition
    \begin{align*}
        \Db(Y) = \langle &\cO_Y(-1,-1,-1),\cO_Y(0,-1,-1), \cO_Y(-1,0,-1), \cO_Y(-1,-1,0), \\
&\cO_Y(0,0,0), \cO_Y(1,0,0), \cO_Y(0,1,0), \cO_Y(0,0,1) \rangle .
    \end{align*}
\end{lemma}

\begin{proof}
    We have 
    \begin{align*}
    \Db(Y) = \langle &\Db(\bP^1) \boxtimes \Db(\bP^1) \boxtimes \Db(\bP^1) \rangle \\
    = \langle  &\cO_Y(i,j,k) \rangle_{i,j,k \in \{-1, 0 \}} \\
    = \langle &\cO_Y(-1,-1,-1) , \cO_Y(0,-1,-1), \cO_Y(-1,0,-1) , \cO_Y(-1,-1,0) , \\
    &\cO_Y(-1,0,0), \cO_Y(0,-1,0), \cO_Y(0,0,-1), \cO_Y(0,0,0)  \rangle
\end{align*}
    by \cite[Theorem 5.8]{KuzBC}. Now apply Lemma \ref{lem:3-1_mutations} to get the desired semiorthogonal decomposition.
\end{proof}

 Set 
\begin{equation*}
    \mathcal{E}_1 \coloneqq \cO_Y(0,0,0), \quad \mathcal{E}_2 \coloneqq \cO_Y(1,0,0), \quad \mathcal{E}_3 \coloneqq \cO_Y(0,1,0), \quad \mathcal{E}_4 \coloneqq \cO_Y(0,0,1) 
\end{equation*}
with $\underline{\cE} = \{ \cE_1, \dots, \cE_4 \}$. Also set $\mathbf{1} = \cO_Y(1,1,1)$. Then clearly $\underline{\cE}$ is an exceptional collection in $\Db(Y)$, as is $\{ \underline{\cE}(-\mathbf{1}), \underline{\cE} \} = \Db(Y)$ by Lemma \ref{lem:3-1 D(Y) SOD}. Then by Proposition \ref{prop: double cover SODs IK}, we get following semiorthogonal decomposition.

\begin{lemma} \label{lem:3-1 A_X SOD}
    We have the semiorthogonal decomposition
    \begin{align*}
        \Db(X) = \langle \cA_X ,  \cO_X(0,0,0), \cO_X(1,0,0), \cO_X(0,1,0), \cO_X(0,0,1)  \rangle .
    \end{align*}
\end{lemma}
From Lemma \ref{lem:3-1 A_X SOD}, we obtain the first semiorthogonal decomposition of $\Db(X)^{\mu_2}$.
\begin{lemma} \label{lem:3-1 second equivariant sod}
    We have the semiorthogonal decomposition 
    \begin{align*}
        \Db(X)^{\mu_2} = \langle  &\cA_X^{\mu_2}, \mathcal{O}_X(0,0,0) \rho_0, \mathcal{O}_X(1,0,0) \rho_0, \mathcal{O}_X(0,1,0) \rho_0, \mathcal{O}_X(0,0,1) \rho_0 \\
        & \mathcal{O}_X(0,0,0)\rho_1, \mathcal{O}_X(1,0,0) \rho_1, \mathcal{O}_X(0,1,0) \rho_1, \mathcal{O}_X(0,0,1) \rho_1 \rangle .
    \end{align*}
\end{lemma}
\begin{proof}
    Apply Theorem \ref{thm:G action commutes through SOD} to the semiorthogonal decomposition from Lemma \ref{lem:3-1 A_X SOD}.
\end{proof}

The second semiorthogonal decomposition of $\Db(X)^{\mu_2}$ comes from \Cref{thm: KP thm 4.1}.
\begin{lemma} \label{lem:3-1 cover SOD}
    We have the semiorthogonal decomposition 
    \begin{align*}
        \Db(X)^{\mu_2} = \langle &f_0^* \Db(Y), j_{0*} \Db(Z) \rangle \\
        = \langle &\cO_X(-1,-1,-1) \rho_0 , \cO_X(0,-1,-1) \rho_0, \cO_X(-1,0,-1) \rho_0 , \cO_X(-1,-1,0) \rho_0 \\
    &\cO_X(-1,0,0) \rho_0, \cO_X(0,-1,0) \rho_0, \cO_X(0,0,-1) \rho_0, \cO_X(0,0,0) \rho_0 , j_{0*} \Db(Z) \rangle .
    \end{align*}
\end{lemma}

\begin{proof}
    This is an instance of \Cref{thm: KP thm 4.1} combined with Lemma \ref{lem:3-1 D(Y) SOD}.
\end{proof}

\begin{proposition} \label{prop:3-1_equivariant_Ku}
    We have 
    \begin{equation*}
        \cA_X^{\mu_2} \simeq \Db(Z) .
    \end{equation*}
\end{proposition}

\begin{proof}
    We apply Lemma \ref{lem:SODs and Serre functors} to the semiorthogonal decomposition from Lemma \ref{lem:3-1 cover SOD} with 
    \begin{equation*}
        \cA_1 = \langle \cO_X(-1,-1,-1) \rho_0 , \cO_X(0,-1,-1) \rho_0, \cO_X(-1,0,-1) \rho_0, \cO_X(-1,-1,0) \rho_0 \rangle 
    \end{equation*}
    and 
    \begin{equation*}
        \cA_2 = \langle \cO_X(-1,0,0) \rho_0, \cO_X(0,-1,0) \rho_0, \cO_X(0,0,-1) \rho_0, \cO_X(0,0,0) \rho_0 , j_{0*} \Db(Z) \rangle .
    \end{equation*}
    We also write 
    \[
    \cA_2' \coloneqq {}^{\perp}(j_{0*}\Db(Z)) \subset \cA_2.
    \]
    Note that $S_{\Db(X)^{\mu_2}}(-) = - \otimes \cO_X(-1,-1,-1) \rho_1$. We now mutate $\cA_2$ to the left along $\cA_1$, and then mutate $j_{0*}\Db(Z)$ to the left along $\cA_2'$: 
    \begin{align*}
        \Db(X)^{\mu_2} = \langle &\cA_2, S_{\Db(X)^{\mu_2}}^{-1}(\cA_1)  \rangle \\
        = \langle &\cO_X(-1,0,0) \rho_0, \cO_X(0,-1,0) \rho_0, \cO_X(0,0,-1) \rho_0, \cO_X(0,0,0) \rho_0 , j_{0*} \Db(Z) , \\
        &\mathcal{O}(0,0,0)\rho_1, \mathcal{O}(1,0,0) \rho_1, \mathcal{O}(0,1,0) \rho_1, \mathcal{O}(0,0,1) \rho_1    \rangle \\
        = \langle &\bL_{\cA_2'} j_{0*} \Db(Z) , 
        \cO_X(-1,0,0) \rho_0, \cO_X(0,-1,0) \rho_0, \cO_X(0,0,-1) \rho_0, \cO_X(0,0,0) \rho_0 \\
        &\mathcal{O}_X(0,0,0)\rho_1, \mathcal{O}_X(1,0,0) \rho_1, \mathcal{O}_X(0,1,0) \rho_1, \mathcal{O}_X(0,0,1) \rho_1   \rangle \\
        = \langle &\bL_{\cA_2'} j_{0*} \Db(Z) , 
        \mathcal{O}_X(0,0,0) \rho_0, \mathcal{O}_X(1,0,0) \rho_0, \mathcal{O}_X(0,1,0) \rho_0, \mathcal{O}_X(0,0,1) \rho_0 \\
        &\mathcal{O}_X(0,0,0)\rho_1, \mathcal{O}_X(1,0,0) \rho_1, \mathcal{O}_X(0,1,0) \rho_1, \mathcal{O}_X(0,0,1) \rho_1   \rangle
    \end{align*}
    For the last equality, we have applied Lemma \ref{lem:3-1_mutations} three times. The right-hand-sides of the above and the semiorthogonal decomposition from Lemma \ref{lem:3-1 second equivariant sod} are the same, and $\bL_{\cA_2'} $ is a fully faithful embedding, hence we conclude the result.
\end{proof}    

Finally, we show the relation between the Serre functor and the involution on $\cA_X$.
\begin{proposition}\label{prop:3-1 serre involution relationship}
    We have 
    \begin{equation*}
        S_{\cA_X} \simeq \tau_{\cA_X}[2] .
    \end{equation*}
\end{proposition}

\begin{proof}
    Since $\Db(Y) = \langle \underline{\cE}(-\mathbf{1}), \underline{\cE} \rangle $ by Lemma \ref{lem:3-1 D(Y) SOD}, the collection $\{ \underline{\cE}(-\mathbf{1}), \underline{\cE} \}$ is full in $\Db(Y)$. Thus, by Proposition \ref{prop: IK prop 3.2}, we get the desired isomorphism of functors. 
\end{proof}

\section{Categorical Torelli Theorems} \label{sec:categorical torelli theorems}
We are now ready to put all of the previous results together to prove \Cref{thm: intro main theorem}.
Recall that a Fano threefold $X$ in one of the three families under consideration is called $Z$-general if the N\'eron--Severi lattice of the branch divisor $Z$ is isometric to $\Lts, \Lte$, or $\Lto$, see Definition \ref{def: Z general}. 

\begin{theorem} \label{thm: main theorem}
    Let $X$ be a $Z$-general Fano variety in one of the families $\mathcal{X}_{2\text{-}6(b)}, \mathcal{X}_{2\text{-}8}$, and $\mathcal{X}_{3\text{-}1}$,
    and let $X'$ be a Fano variety in the same family. Then, if there exists an 
    equivalence between the Kuznetsov components $\cA_X$ and $\cA_{X'}$, it follows that $X$ is isomorphic to $X'$:
    \[
    \cA_X \simeq \cA_{X'} \implies X\simeq X'.
    \]
\end{theorem}
\begin{proof}
    Let $Z$ and $Z'$ be the associated K3 surfaces of $X$ and $X'$. 
    By Theorem \ref{thm: any equivalence lifts to equivariant categories}, the equivalence $\Phi$ lifts to an equivalence of equivariant categories $\Phi^{\mu_2}\colon \cA_X^{\mu_2}\simeq \cA_{X'}^{\mu_2}$.
    By Theorem \ref{thm: equivariant kuznetsov components are the branch loci}, this induces an equivalence $\Db(Z)\simeq \Db(Z')$.
    This equivalence is Fourier--Mukai by \cite{OrlovK3}. 
    In particular, $Z$ and $Z'$ are Fourier--Mukai partners. 
    By Theorem \ref{thm: the k3s have no FM partners}, this means that we must have $Z\simeq Z'$. 
    Finally, we conclude that $X\simeq X'$ by Theorem \ref{thm: branch determines Fano}.
\end{proof}

\begin{theorem}\label{thm: categorical torelli for verra fourfolds}
    Let $V_4, V_4'$ be very general Verra fourfolds and suppose that there is an equivalence $\cA_{V_4} \simeq \cA_{V_4'}$. Then it follows that $V_4$ is isomorphic to $V_4'$:
    \begin{equation*}
        \cA_{V_4} \simeq \cA_{V_4'} \implies V_4 \simeq V_4' .
    \end{equation*}
\end{theorem}

\begin{proof}
    Let $V_3$ and $V_3'$ be the branch loci of the double covering maps $V_4\to \bP^2\times \bP^2$ and $V_4' \to \bP^2 \times \bP^2$, i.e. $V_3$ and $V_3'$ are smooth ordinary Verra threefolds. 
    Then we have 
    \[
    \Db(S_1,\alpha_1) \simeq \Db(S_2,\alpha_2) \simeq \cA_{V_4} \simeq \cA_{V_4'} \simeq \Db(S_1',\alpha_1') \simeq \Db(S_2',\alpha_2').
    \]
    By our generality assumptions on $V_4$ and $V_4'$, the K3 surfaces $S_i$ and $S_i'$, where $i = 1,2$, each have Picard rank $1$, hence we may apply Proposition \ref{prop: the derived category of the verra twisted K3 determines the Verra threefold} and Lemma \ref{lem: brauer classes match}  to conclude that for $i,j\in \left\{1,2\right\}$, we have $V_3 \simeq V_{\alpha_i} \simeq V_{\alpha_j'} \simeq V_3'$. Thus, by Proposition \ref{prop: branch determines cover V4}, we have $V_4 \simeq V_4'$.
\end{proof}

\subsection{Open problems}
\subsubsection{Special Verra fourfolds}
Our proof of the categorical Torelli theorem for Verra fourfolds uses the generality assumption in two ways:
\begin{enumerate}
    \item To ensure that the two associated K3 surfaces of the Verra fourfold have Picard rank $1$ (needed in the proof of Lemma \ref{lem: the twisted k3s have few partners});
    \item To ensure that the two associated K3 surfaces of the Verra fourfold are not isomorphic (needed in the proof of Proposition \ref{prop: the derived category of the verra twisted K3 determines the Verra threefold}).
\end{enumerate}

For special Verra threefolds, both of these properties may fail, and it is an important problem to study the potential Fourier--Mukai partners of these Verra fourfolds. For example, using the similarity between Verra fourfolds and cubic fourfolds, we may ask the following two questions:
\begin{question} \leavevmode
    \begin{itemize}
    \item[(a)] If two Verra fourfolds have equivalent Kuznetsov components, are they birational?
    \item[(b)] Is a Verra fourfold rational if and only if its Kuznetsov component is equivalent to an untwisted K3 surface?
\end{itemize}
\end{question}

Another motivation to study Fourier--Mukai partners of Verra fourfolds is $L$-equivalence, since the associated K3 surfaces of a Verra fourfold are well-known for being $L$-equivalent under certain conditions, cf. \cite[\S2.6.2]{KS}, but not isomorphic \cite{KKM}. In light of the conjectural relationship between $L$- and $D$-equivalence for K3 surfaces \cite{KS, Mei}, as well as the analogous results for cubic fourfolds \cite{BL, MM}, we also ask the following question:
\begin{question}
    If two Verra fourfolds are $L$-equivalent, are their Kuznetsov components equivalent?
\end{question}

\subsubsection{Family 2-8(b)}
Recall from Remark \ref{rmk: why not 2-8(b)} that Family 2-8 splits into two subfamilies, called 2-8(a) and 2-8(b). 
Theorem \ref{thm: main theorem} holds for $Z$-general Fano threefolds in Family 2-8. 
These are all members of Family 2-8(a). A Fano threefold in Family 2-8 is a member of Family 2-8(b) if the scheme-theoretic intersection $Z\cap E$ is singular but reduced. 
In the smooth case, the exceptional divisor $e \in \NS(Z)$ always satisfies $e^2 = 2$. 
However, if $Z\cap E$ is not smooth, the N\'eron--Severi lattice of $Z$ depends on the singularities of $Z\cap E$. 
Thus, to prove a categorical Torelli theorem for very general members of Family 2-8(b), more N\'eron--Severi lattices need to be considered. 

\subsubsection{The remaining double covers} \label{sec: remaining double covers}
Finally, as we already noted in the introduction, there are only two remaining families of Fano threefolds of Picard rank $\rho \geq 2$ which are listed as double covers on Fanography \cite{fanography} for which a categorical Torelli theorem has not been proved, namely Families 2-2 and 2-18. 
    \begin{enumerate}
        \item \textbf{Family 2-2:} These are double covers of $\bP^1 \times \bP^2$ branched in a divisor of bidegree $(2,4)$. The approach of the current paper is not possible, since the canonical bundle of $Z$ is no longer trivial. In fact, $K_Z \simeq \cO_Z(0,1)$, so $Z$ is of general type. By the trichotomy discussed at the end of \cite[Section 3]{DJR}, one might expect the positivity of $K_Z$ to mean that the equivariant Kuznetsov component should contain a copy of $\Db(Z)$. One could then apply the techniques of \cite{DJR} to show that an equivalence of Kuznetsov components implies a Hodge isometry between the middle cohomologies of the respective branch divisors. However, we are unaware of a Torelli theorem for bidegree $(2,4)$ divisors in $\bP^1 \times \bP^2$.
        \item \textbf{Family 2-18:} These are double covers of $\bP^1 \times \bP^2$ branched in a divisor of bidegree $(2,2)$. In this case, $K_Z \simeq \cO_Z(0,-1)$, i.e. $Z$ is rational, but not Fano. By the trichotomy from \cite[Section 3]{DJR}, the negativity of $K_Z$ might lead one to expect $\Db(Z)$ to contain the equivariant Kuznetsov component.
    \end{enumerate}

\bibliographystyle{alpha}
{\small{\bibliography{references}}}
\end{document}